\documentclass[11pt]{amsart}
\usepackage{amsmath}
\usepackage{amssymb}
\usepackage{graphicx}
\usepackage{url}
\usepackage{multirow}
\usepackage{algorithm}
\usepackage{algorithmic}
\usepackage{float}
\usepackage{comment}
\usepackage[font=small,labelfont=bf]{caption}
\usepackage{subcaption}
\usepackage{tabularx}
\usepackage{color}
\usepackage{hyperref}

\newtheorem{theorem}{Theorem}[section]

\theoremstyle{definition}
\newtheorem{definition}[theorem]{Definition}
\newtheorem{example}[theorem]{Example}
\newtheorem{remark}[theorem]{Remark}

\numberwithin{equation}{section}

\title[Higher degree inexact model] {Higher Degree Inexact Model for Optimization Problems}

\author[M.~Alkousa]{Mohammad S. Alkousa}
\address[M.~S.~Alkousa]{Moscow Institute of Physics and Technology, Russia.}
\email{\tt mohammad.alkousa@phystech.edu}

\author[F.~Stonyakin]{Fedor~S.~Stonyakin}
\address[F.~S.~Stonyakin]{Moscow Institute of Physics and Technology and V. I. Vernadsky Crimean Federal University, Russia.}
\email{{\tt fedyor@mail.ru}}

\author[A.~Gasnikov]{Alexander~V.~Gasnikov}
\address[A.~V.~Gasnikov]{Moscow Institute of Physics and Technology, Innopolis University and  Caucasus Mathematical Center of Adygh State University, Russia.}
\email{{\tt gasnikov@yandex.ru}}

\author[A.~Abdo]{Asmaa~M.~Abdo}
\address[A.~M.~Abdo]{Mathematics Department, Faculty of Science, Damascus University, Damascus, Syria.}
\email{{\tt asmaa.abdo@damascusuniversity.edu.sy}}

\author[M.~Alcheikh]{Mohammad~M.~Alcheikh}
\address[M.~M.~Alcheikh]{Mathematics Department, Faculty of Science, Damascus University, Damascus, Syria.}
\email{{\tt mohammad.alcheikh@damascusuniversity.edu.sy}}

\keywords{Inexact model, Inexact oracle, Adaptive gradient method, Fast gradient method, Universal Fast Gradient Method, Convex optimization, Saddle point, Variational inequality.}

\begin{document}

\begin{abstract}
In this paper, it was proposed a new concept of the inexact higher degree $(\delta, L, q)$-model of a function that is a generalization of the inexact $(\delta, L)$-model \cite{Gasnikov2019fast}, $(\delta, L)$-oracle \cite{Devolder2013smooth} and $(\delta, L)$-oracle of degree $q \in [0,2)$ \cite{Nabou2024qoracle}. Some examples were provided to illustrate the proposed new model. Adaptive inexact gradient and fast gradient methods for convex and strongly convex functions were constructed and analyzed using the new proposed inexact model. A universal fast gradient method that allows solving optimization problems with a weaker level of smoothness, among them non-smooth problems was proposed. For convex optimization problems it was proved that the proposed gradient and fast gradient methods could be converged with rates $O\left(\frac{1}{k} + \frac{\delta}{k^{q/2}}\right)$ and $O\left(\frac{1}{k^2} + \frac{\delta}{k^{(3q-2)/2}}\right)$, respectively. For the gradient method, the coefficient of $\delta$ diminishes with $k$, and for the fast gradient method, there is no error accumulation for $q \geq 2/3$. It proposed a definition of an inexact higher degree oracle for strongly convex functions and a projected gradient method using this inexact oracle. For variational inequalities and saddle point problems, a higher degree inexact model and an adaptive method called Generalized Mirror Prox to solve such class of problems using the proposed inexact model were proposed. Some numerical experiments were conducted to demonstrate the effectiveness of the proposed inexact model, we tested the universal fast gradient method to solve some non-smooth problems with a geometrical nature.
\end{abstract}

\maketitle

\section{Introduction} \label{sec:Introduction}

With the increase in the number of applications that can be modeled as large (or huge) scale optimization problems (some of such applications arising in machine learning, deep learning, data science, control, signal processing, statistics \cite{Beck2017,Bottou2018machine}, and so on)  first-order methods, which require low iteration cost as well as low memory storage, have received much interest over the past few decades to solve the optimization problems when accuracy requirements are not high. 

When the objective function is smooth, the simplest numerical schemes to be considered are the gradient method and its variants. It is known that these methods converge to a solution of the problem with rate $O(1/k)$, where $k$ is the counter of iterations.  However, it is well-known that in the black-box framework \cite{Nemirovskii1983Complexity}, the optimal convergence rate for first-order methods is $O(1/k^2)$,  such optimal methods (which are called Fast Gradient Methods (FGM)) have been developed for many different classes of problems since 1983 \cite{Nesterov1983,Nesterov_book,Nesterov1988approach}. 

These methods (i.e., gradient-type methods) are constructed using some model of the objective function $f$ at the current iterate $x_k$. This can be a quadratic model based on the $L$-smoothness of $f$, i.e.,
\begin{equation}\label{eq:quadr_model}
    f(x_k) + \langle \nabla f(x_k), x-x_k \rangle + \frac{L}{2}\|x-x_k\|^2,
\end{equation}
where $\|\cdot\|$ denotes the standard Euclidean norm.
The scheme of the gradient method is obtained by the minimization of this model \cite{Nesterov_book}.  More general models are constructed based on regularized second-order Taylor expansion \cite{nesterov2006cubic} or other Taylor-like models \cite{drusvyatskiy2019nonsmooth} as well as other objective surrogates \cite{mairal2013optimization}. Another example is the conditional gradient method \cite{frank1956algorithm}, where a linear model of the objective is minimized on every iteration. Adaptive choice of the parameter of the model with provably small computational overhead was proposed in \cite{nesterov2006cubic} and applied to first-order methods in \cite{Nesterov2015Universal,nesterov2013gradient}. Recently, first-order optimization methods were generalized to the so-called relative smoothness framework \cite{bauschke2016descent,lu2018relatively,ochs2017non}, where $\frac{1}{2} \|x-x_k\|^2$ in the quadratic model \eqref{eq:quadr_model} for the objective is replaced by general Bregman divergence.

Standard analysis of first-order methods assumes the availability of exact first-order information. Namely, the oracle must provide at each given point the exact values of the function and its gradient. However, in many problems, including those obtained by smoothing techniques \cite{Nesterov2005Smooth}, the objective function and its gradient are computed by solving another auxiliary optimization problem. In practice, we are often only able to solve these subproblems approximately. Hence, in that context, numerical methods solving the outer problem are provided with inexact first-order information. This led us to investigate the behavior of first-order methods working with an inexact oracle. Optimization algorithms with inexact first-order oracles are well-studied in the literature \cite{Cohen2018Acceleration,Devolder2013smooth,Devolder2013strongly,Devolder2013disser,Dvurechensky2016Stochastic,Dvurechensky2022gradient,Aspremont2008Smooth}.

In \cite{Devolder2013smooth}, authors introduced the so-called inexact first-order $(\delta, L)$-oracle for the function $f$, at a given point $y \in Q$, where $Q$ is a convex set, i.e., one can compute a pair $(f_{\delta, L}(y), g_{\delta, L}(y))$, such that
\begin{equation}\label{ew21wer5}
    0 \leq f(x)-\left(f_{\delta, L}(y)+\left\langle g_{\delta, L}(y), x-y\right\rangle\right) \leq \frac{L}{2}\|x-y\|^2+\delta,  \quad \forall x \in Q, 
\end{equation}
and they considered a classical (primal) gradient method and a fast gradient method (FGM) with inexact oracle. The convergence rates for these methods are $O\left(\frac{1}{k} + \delta\right)$ and $O\left(\frac{1}{k^2} + k \delta\right)$, respectively. One can notice that for the classical (non-accelerated) method, the objective function accuracy decreases with $k$ and asymptotically tends to $\delta$, while in the accelerated scheme, there is an error accumulation.

To generalize the concept of $(\delta, L)$-oracle \eqref{ew21wer5}, in \cite{Gasnikov2019fast} a new concept of $(\delta, L)$-model of a function was proposed. That is, the pair $(f_{\delta, L} (y), \psi_{\delta, L} (x, y))$ is called a $(\delta, L)$-model of the function $f(x)$ at the given point $y \in Q$, if it holds the following inequality
\begin{equation}\label{inex_model_Gasnikov}
    f(x) - \left( f_{\delta, L} (y) + \psi_{\delta, L} (x, y)\right) \leq \frac{L}{2} \|x - y\|^2 + \delta, \quad \forall x \in Q,
\end{equation}
and $\psi_{\delta, L} (x, y)$  is convex in $x$, satisfies $\psi_{\delta, L}(x,x) = 0$  for all $x \in Q$. Note that this concept (i.e., \eqref{inex_model_Gasnikov}) generalizes the $(\delta, L)$-oracle concept \eqref{ew21wer5}, where it is enough to take $\psi_{\delta, L} (x, y) = \left\langle g_{\delta, L}(y), x-y\right\rangle.$ Within this concept, the gradient descent and fast gradient descent methods are constructed and it was shown that constructs of many known methods (composite methods, level methods, conditional gradient, and proximal methods) are particular cases of the methods proposed in \cite{Gasnikov2019fast}. A more generalization of the results in \cite{Gasnikov2019fast} was conducted in \cite{Stonyakin2021OMS}, where authors presented a unified view on inexact models for optimization problems, variational inequalities, and saddle-point problems. In \cite{Gasnikov2019fast,Stonyakin2021OMS} it was proved convergence rates of many gradient-type methods, which cover the known results by $(\delta, L)$-oracle, i.e., $O\left(\frac{1}{k} + \delta\right)$ and $O\left(\frac{1}{k^2} + k \delta\right)$. 

Recently, in \cite{Nabou2024qoracle} it was introduced the concept of the inexact first-order oracle of degree $q \in [0,2)$ for minimization (possibly non-convex) problems. With this concept (see Definition \ref{def_neq_inex_model_non_convex} and Remark \ref{remark_yassine_special}), and for convex optimization problems it was proved that the classical inexact gradient method and inexact fast gradient method can be converged with rates $O\left(\frac{1}{k} + \frac{\delta}{k^{q/2}}\right)$ and $O\left(\frac{1}{k^2} + \frac{\delta}{k^{(3q-2)/2}}\right)$, respectively. Note that for the inexact gradient method, the coefficient of $\delta$ diminishes with $k$, and for the inexact fast gradient method there is no error accumulation for $q \geq 2/3$.

In this paper, we generalize the results of \cite{Nabou2024qoracle}, by proposing a higher degree inexact model (see Definitions \ref{def_neq_inex_model_non_convex}, and \ref{def_neq_inex_model_convex}). This model is also a generalization of the inexact model proposed in \cite{Gasnikov2019fast,Stonyakin2021OMS} (see also \eqref{inex_model_Gasnikov}). We proposed adaptive gradient and fast gradient methods (Algorithms \ref{alg_AdapGM} and \ref{alg2_FGM}) with the proposed inexact higher degree model for smooth convex functions as well as for strongly convex functions by using the technique of restarts. With the proposed inexact higher degree model we construct a universal fast gradient method for solving problems with a weaker level of smoothness. In addition, for the convex functions we defined a higher degree inexact oracle (see Definition \ref{def_q_oracle_strongly}), which generalizes the Devolder-Glineur-Nesterov $(\delta, L, \mu)$-oracle proposed in \cite{Devolder2013strongly}. We also adapted the proposed inexact higher degree model for variational inequalities and saddle point problems and proposed an adaptive algorithm called Generalized Mirror Prox for variational inequalities with $(\delta, L, q)$-model (see Algorithm \ref{Alg:UMPModel}).

\subsection{ Contributions}
\begin{itemize}
    \item We introduce an inexact higher degree $(\delta, L, q)$-model for convex and non-convex optimization problems and obtain convergence the rate for an adaptive inexact gradient method for optimization problems with this model.

    \item We obtain convergence rates for an adaptive inexact fast gradient method (FGM) for optimization problems with a $(\delta, L, q)$ model. Using the technique of restarts of FGM with the proposed higher degree inexact model we obtain a convergence rate of the restarted method for strongly convex problems. Also, using the FGM we construct a universal fast gradient method UFGM.

    \item We introduce an inexact higher degree $(\delta, L, \mu, q)$-oracle for strongly convex optimization problems and obtain a convergence rate of the inexact gradient method with this oracle without using the technique of restarting any other algorithms.

    \item  We introduce an inexact higher degree $(\delta, L, q)$-model for variational inequalities and saddle-point problems and obtain convergence rates for adaptive versions of the Generalized Mirror–Prox algorithm for problems with this model.

    \item We conduct some numerical experiments for testing the proposed UFGM with the proposed inexact model for non-smooth optimization problems: the best approximation problem and the Fermat-Torricelli-Steiner problem.     
\end{itemize}

\subsection{Paper Organization}
The paper consists of an introduction and 9 main sections.  In Sect. \ref{sect_basics} we mentioned the statement of the considered problem and the connected fundamental concepts.  Sect. \ref{sect:def_model_q} devoted to the definitions of the inexact higher degree model for non-convex and convex functions, we provided some examples to illustrate the proposed definitions of the new inexact model. In Sect. \ref{sect_adaptiveGM} we proposed an adaptive gradient method with the proposed inexact higher degree model for smooth convex functions.  In Sect.  \ref{sec:FGM} we proposed an adaptive fast gradient method with the proposed inexact higher degree model for smooth convex functions and strongly convex functions. We also proposed a universal fast gradient, which allows us to solve optimization problems with a weaker level of smoothness, among them non-smooth problems. In Sect. \ref{sec:higher_deg_oracle}, we defined the higher degree inexact $(\delta, L, \mu, q)$-oracle for convex functions and proposed a projected gradient method using this proposed inexact oracle. In Sections \ref{Sect:VI} and \ref{Sect:SPP}, we defined an inexact higher degree model for variational inequalities and saddle point problems, respectively. We proposed an adaptive method (called Generalized Mirror Prox) with the new inexact model and we analyzed this method for variational inequalities and saddle point problems. In Sect. \ref{sect_numerical} we presented the results of some numerical experiments, these results demonstrate the effectiveness of the proposed inexact model. We tested the universal fast gradient method to solve two non-smooth problems (the best approximation problem and the Fermat-Torricelli-Steiner problem). Section \ref{sect:Conclusion} concludes the paper, in which we summarize the concluded results.   

\section{Problem Statement and Fundamentals }\label{sect_basics}

In this paper, we consider the following optimization problem
\begin{equation}\label{main_min_prob}
    \min_{x \in Q} f(x),
\end{equation}
where $Q \subseteq \mathbb{R}^n$ is a convex compact set and  $f: Q \longrightarrow \mathbb{R}$ is a smooth function, i.e., there exist $L>  0$, such that  
\begin{equation}\label{smoothness_cond}
    f(x) \leq f(y)+ \left\langle\nabla f(y), x - y\right\rangle + \frac{L}{2} \|x - y\|^2, \quad \forall x, y \in Q,
\end{equation}
(here and everywhere in the paper we use $\|\cdot\|$ to denote the standard Euclidean norm) or equivalently 
\begin{equation}\label{smoothness_cond2}
    \|\nabla f(x) - \nabla f(y)\| \leq L_f \|x - y\|, \quad \forall x, y \in Q.
\end{equation}

The function $f: Q \longrightarrow \mathbb{R}$ is $\mu$-strongly convex, for some $\mu >  0$,  if it holds
\begin{equation}\label{eq:str_cvx}
    f(x) \geq f(y) + \langle \nabla f(y),x - y \rangle + \frac{\mu}{2}\|x - y\|^2, \quad  \forall x, y \in Q. 
\end{equation}

When $\mu = 0$ in \eqref{eq:str_cvx}, the function $f$ will be a convex.

\section{Inexact $(\delta, L, q)$-model: Definitions and Examples}\label{sect:def_model_q}

\begin{definition}\label{def_neq_inex_model_non_convex}
Let $\delta \geq 0, L>$ and $q \in [0,2)$. The pair $(f_{\delta, L, q} (y), \psi_{\delta, L ,  q} (x, y))$ is called a $(\delta, L, q)$-model of degree $q$ of the function $f(x)$ at the given point $y \in Q$, if it holds the following inequality
\begin{equation}\label{ineq_inex_model_1}
    f(x) - \left( f_{\delta, L, q} (y) + \psi_{\delta, L, q} (x, y)\right) \leq \frac{L}{2} \|x - y\|^2 + \delta \|x - y\|^q, \quad \forall x \in Q,
\end{equation}
and $\psi_{\delta, L,  q} (x, y)$  is convex in $x$, satisfies $\psi_{\delta, L ,  q}(x,x) = 0$  for all $x \in Q$. 

For the sake of brevity, we will say that $\psi_{\delta, L ,  q} (x, y))$ is a $(\delta, L, q)$-model of degree $q$ of the function $f(x)$ at the given point $y \in Q$, instead of the pair $(f_{\delta, L, q} (y), \psi_{\delta, L ,  q} (x, y))$. 
\end{definition}

When the function $f$  is convex, we consider the following modified definition. 

\begin{definition}\label{def_neq_inex_model_convex}
Let $f: Q \longrightarrow \mathbb{R}$ be a convex function, $\delta \geq 0, L>$ and $q \in [0,2)$. The pair $(f_{\delta, L, q} (y), \psi_{\delta, L ,  q} (x, y))$ is called a $(\delta, L, q)$-model of degree $q$ of the function $f(x)$ at the given point $y \in Q$, if it holds the following inequalities
\begin{equation}\label{ineq_inex_model_convex}
    0 \leq f(x) - \left( f_{\delta, L, q} (y) + \psi_{\delta, L, q} (x, y)\right) \leq \frac{L}{2} \|x - y\|^2 + \delta \|x - y\|^q, \quad \forall x \in Q,
\end{equation}
and $\psi_{\delta, L,  q} (x, y)$  is convex in $x$, satisfies $\psi_{\delta, L ,  q}(x,x) = 0$  for all $x \in Q$. 
\end{definition}

\begin{remark}\label{remark_yassine_special}
Note that the inexact $(\delta, L)$-model of a function, which is considered in \cite{Gasnikov2019fast,Stonyakin2021OMS} is a $(\delta, L, 0)$-model in the sense of  Definitions \ref{def_neq_inex_model_non_convex} and \ref{def_neq_inex_model_convex}.

Now, let $\delta \geq 0, L>, q \in [0, 2)$, and  $y \in Q$ is a given point. Let us set  $\psi_{\delta, L ,  q} (x, y) = \left\langle g_{\delta, L, q} (y),  x - y \right\rangle$, where $g_{\delta, L, q}(y) \in \mathbb{R}^n$, then, form \eqref{def_neq_inex_model_non_convex} (or \eqref{ineq_inex_model_convex}), we have 
\[
    f(x) - \left( f_{\delta, L, q} (y) + \left\langle g_{\delta, L, q} (y),  x - y \right\rangle\right) \leq \frac{L}{2} \|x - y\|^2 + \delta \|x - y\|^q, \quad \forall x \in Q. 
\]

Thus, we get the definition of the inexact first-order $(\delta, L)$-oracle of degree $q$, which was proposed in \cite{Nabou2024qoracle}.
\end{remark}

\begin{remark}\label{rem_ordin_model}
Since (see \cite{Nabou2024qoracle})
\begin{equation}\label{eq_1242}
    \delta \|x - y\|^q \leq \frac{q \rho \|x - y\|^2}{2} + \frac{(2-q) \delta^{\frac{2}{2-q}}}{2 \rho^{\frac{q}{2-q}}}, \quad \forall \rho >0, q \in [0,2),
\end{equation}
then from \eqref{def_neq_inex_model_non_convex} (or \eqref{ineq_inex_model_convex}), we get the following
\begin{equation}\label{spesial_case_model}
    0 \leq f(x) - \left( f_{\delta, L, q} (y) + \psi_{\delta, L, q} (x, y)\right) \leq \frac{\widehat{L}}{2} \|x - y\|^2 + \widehat{\delta},
\end{equation}
where
\begin{equation}\label{parameters_spec}
   \widehat{L} = L + q \rho \quad \text{and} \quad \widehat{\delta} = \frac{(2-q) \delta^{\frac{2}{2-q}}}{2 \rho^{\frac{q}{2-q}}}. 
\end{equation}

Thus, $\psi_{\delta,L, q} (x, y)$ which is given in Definitions \ref{def_neq_inex_model_non_convex} and \ref{def_neq_inex_model_convex}, represents a $(\widehat{\delta}, \widehat{L})$-model of the function $f(x)$ at a given point $y \in Q$, for any $x \in Q$, in a sense of \cite{Gasnikov2019fast}. 
\end{remark}

Next, we list some examples to illustrate the proposed inexact higher degree model in Definitions \ref{def_neq_inex_model_non_convex} and \ref{def_neq_inex_model_convex}.  

\begin{example}\label{ex:relative_error} (\textit{Relative inexactness of the gradient}).

In many applications, instead of access to the exact gradient $\nabla f(x)$ of the objective function $f$ at a point $x \in Q$, we access only to its inexact approximation $\widetilde{\nabla} f(x)$. Typical examples of such applications include gradient-free (or zeroth-order) methods which use a gradient estimator based on finite differences \cite{Berahas2021theoretical,Conn2009Introduction,Risteski2016Algorithms}, and optimization problems in infinite-dimensional spaces related to inverse problems \cite{Kabanikhin2011,Gasnikov2017Hilbert}. 

One of the most popular definitions of the gradient inexactness in practice is  \cite{Polyak1987book}
\begin{equation}\label{relative_noise}
    \left\|\widetilde{\nabla} f(x) - \nabla f(x)\right\| \leq \alpha \|\nabla f(x)\|, \quad \text{ for some } \alpha \in (0, 1). 
\end{equation}

This means, that an additive error in the gradient is proportional to the gradient norm, rather than being globally bounded by some small quantity. It is called the relative inexactness of the gradient (see Example \ref{ex1_inexact} for the absolute inexactness of the gradient).

Let $y \in Q$ be a given point, from \eqref{smoothness_cond} and \eqref{relative_noise}, for any $ x \in Q$, we have 
\begin{align*}
    f(x) & \leq f(y) + \left\langle \widetilde{\nabla}   f(y), x - y \right\rangle + \frac{L_f}{2} \left\|x - y \right\|^2  + \left\langle \nabla f(y) - \widetilde{\nabla} f(y),  x  - y \right\rangle
    \\& \leq f(y) + \left\langle \widetilde{\nabla}  f(y), x - y  \right\rangle + \frac{L_f}{2} \|x -y \|^2  + \left\| \widetilde{\nabla} f(y) - \nabla f(y)\right\| \cdot  \| x  - y \|
    \\& \leq f(y) + \left\langle \widetilde{\nabla}  f(y), x - y  \right\rangle + \frac{L_f}{2} \|x - y \|^2 + \alpha \left\| \nabla f(y)\right\| \cdot \|  x - y \|.
\end{align*}
But, from \eqref{relative_noise}, we get the following
\begin{equation*}
    (1-\alpha) \left\|\nabla f(x)\right\| \leq \left\|\widetilde{\nabla} f(x)\right\|  \leq (1+\alpha) \left\|\nabla f(x)\right\|, \quad \forall x \in Q.
\end{equation*}
From this, we have 
\begin{equation*}
    \left\|\nabla f(x)\right\| \leq \frac{1}{1-\alpha} \left\|\widetilde{\nabla} f(x)\right\|, \quad \forall x \in Q.
\end{equation*}

Therefore,  we get 
\begin{align*}
    f(x) & \leq f(y) + \left\langle \widetilde{\nabla}  f(y),  x  - y \right\rangle + \frac{L_f}{2} \|x - y \|^2  + \frac{\alpha}{1-\alpha} \left\|\widetilde{\nabla} f(y)\right\| \cdot \| x - y\|.
\end{align*}
i.e., 
\begin{align*}
    f(x) - \left(f(y) + \left\langle \widetilde{\nabla}  f(y),  x  - y \right\rangle \right)  \leq  \frac{L_f}{2} \|x - y \|^2  + \frac{\alpha}{1-\alpha} \left\|\widetilde{\nabla} f(y)\right \| \cdot \| x - y\|.
\end{align*}

Thus, $\psi_{\delta, L, q} (x,y) =\left\langle \widetilde{\nabla}f(y), x - y \right\rangle $ is a $(\delta, L, q)$-model of the function $f(x)$ at a given point $y \in Q$, with $\delta = \frac{\alpha}{1-\alpha} \left\|\widetilde{\nabla} f(y)\right\|, L = L_f, q = 1$ and $f_{\delta, L, q} (y) = f(y)$. 
\end{example}

The next examples are described by an inexact first-order $(\delta, L)$-oracle of degree $q \in [0, 2)$ in \cite{Nabou2024qoracle}, which can be covered by Definitions  \ref{def_neq_inex_model_non_convex} and \ref{def_neq_inex_model_convex} by a suitable choosing of the pair $\left(f_{\delta,L, q}(y), \psi_{\delta, L, q} (x, y )\right)$. 

\begin{example}\label{ex1_inexact}(\textit{Finite sum optimization and absolute inexactness of the gradient}).

Let $\{f_1, \ldots, f_m\}$ be a set of $m$ functions, each of them is $L_i$-smooth, and let $f(x) = \sum_{i = 1}^{m} f_i(x)$. The function $f$ is an $L_f$-smooth with $L_f=  \sum_{i = 1}^{m} L_i$. The problem \eqref{main_min_prob} with such functions, when $m$ is large enough, is called a finite sum optimization problem. It captures the standard empirical risk minimization problems in machine learning (such as least-squares or logistic regression with e.g. linear predictors or neural networks) \cite{Shwartz2014Understanding} and it appears widely in machine learning applications, including but not limited to deep neural networks, multi-kernel learning \cite{Arora2019Fine,Bottou2018machine,Brutzkus2018SGD,Zou2018Deep}. 

For such problems, it used stochastic gradient-type methods. In these methods, it needs to calculate a so-called mini-bach stochastic gradient of the objective function $f$, i.e.,
\[
    \widetilde{\nabla} f(x) = \frac{1}{|\mathcal{S}|} \sum_{i \in \mathcal{S}} \nabla f_i(x),
\]
where $\mathcal{S} \subseteq \{1, 2, \ldots, m\}$ and $|\mathcal{S}|$ is the size of the mini-bach. 

For this approximation of the exact gradient $\nabla f(x)$, it holds the following inequality
\begin{equation}\label{absolute_error}
   \left \|\nabla f(x) - \widetilde{\nabla} f(x)\right \| \leq \Delta, \quad \text{for some} \quad \Delta >0, 
\end{equation}
with probability at least $1 - \Delta$, if it satisfies $|\mathcal{S}| = O \left( \left(\frac{\Delta^2}{L_f^2} + \frac{1}{m}\right)^{-1}\right)$ (see Lemma 11 in \cite{Agafonov2023}). 

For the function $f$, we have 
\begin{align*}
   & \quad \;\;  f(x) - f(y) - \left \langle\nabla f(y) - \nabla f(y) +  \widetilde{\nabla} f(y)  , x - y \right \rangle  
   \\& = f(x) - f(y) - \left \langle\nabla f(y) , x - y \right \rangle + \left \langle\nabla f(y) -   \widetilde{\nabla} f(y)  , x - y \right \rangle 
   \\& \stackrel{\eqref{smoothness_cond}}{\leq} \frac{L_f}{2} \|x - y\|^2 + \left\|\nabla f(y) -   \widetilde{\nabla} f(y)\right \| \cdot \|x - y\|
   \\& \stackrel{\eqref{absolute_error}}{\leq} \frac{L_f}{2} \|x - y\|^2 + \Delta \|x - y\|.
\end{align*}

Thus, we get
\begin{equation}
    f(x) - \left(f(y) + \left\langle \widetilde{\nabla}f(y), x - y \right\rangle\right) \leq \frac{L_f}{2} \|x - y\|^2 + \Delta \|x - y\|, \quad \forall x \in Q. 
\end{equation}

Therefore, $\psi_{\delta, L, q} (x,y) =\left\langle \widetilde{\nabla}f(y), x - y \right\rangle $ is a $(\delta, L, q)$-model of the function $f(x)$ at a given point $y \in Q$, with $\delta = \Delta, L = L_f, q = 1$ and $f_{\delta, L, q} (y) = f(y)$. 

The inequality \eqref{absolute_error} represents one of the most popular definitions of the gradient inexactness in practice \cite{Polyak1987book}, and it is called the absolute inexactness of the gradient (see Example \ref{ex:relative_error} for another important inexact gradient).  
\end{example}

\begin{remark}
Connecting to the Example \ref{ex1_inexact}, we mention here that in future work, we will study the accelerated and non-accelerated stochastic gradient descent methods for solving smooth (strongly) convex stochastic optimization problems, i.e., finite sum optimization problems, and deriving estimates of the rate of convergence in the proposed higher degree inexact model (see \cite{Dvinskikh2020Accelerated}, where the results here obtained in the model generality with $q = 0$). 
\end{remark}

\begin{example}\label{ex2_inexact}(\textit{Computations at shifted points, see also \cite{Devolder2013smooth}}). 

Let $f: Q \longrightarrow \mathbb{R}$ be an $L_f$-smooth function. Let us assume that we can compute the exact gradient $\nabla f(y)$ at each $y \in Q$,  and the approximated one at a shifted point $\hat{y} \ne y$, such that $\|y - \hat{y}\| \leq \Delta$.  Since $f$ is $L_f$-smooth, we have, for any $x \in Q$
\begin{align*}
    f(x) & \leq f(y) + \langle \nabla f(y), x - y \rangle + \frac{L_f}{2} \| x - y\|^2
    \\& = f(y) + \langle \nabla f(\hat{y}), x - y \rangle +  \frac{L_f}{2} \| x - y\|^2 + \langle \nabla f(y) - \nabla f(\hat{y}), x - y \rangle
    \\& \leq  f(y) + \langle \nabla f(\hat{y}), x - y \rangle +  \frac{L_f}{2} \| x - y\|^2 +  \| \nabla f(y) - \nabla f(\hat{y}) \|_*  \|x - y\|
    \\& \stackrel{\eqref{smoothness_cond2}}{\leq} f(y) + \langle \nabla f(\hat{y}), x - y \rangle +  \frac{L_f}{2} \| x - y\|^2 +  \Delta L_f  \|x - y\|.  
\end{align*}

Thus, we get 
\[
    f(x) - \left(f(y) + \langle \nabla f(\hat{y}), x - y \rangle\right) \leq \frac{L_f}{2} \| x - y\|^2 +  \Delta  L_f \|x - y\|, \quad \forall x \in Q. 
\]

Therefore, $\psi_{\delta, L, q} (x,y) =\langle \nabla f(\hat{y}), x - y \rangle $ is a $(\delta, L, q)$-model of the function $f(x)$ at a given point $y \in Q$, with $\delta = \Delta L_f, L = L_f, q = 1$ and $f_{\delta, L, q} (y) = f(y)$.
\end{example}

\begin{example}\label{ex4_inexact} (\textit{Functions with a weaker level of smoothness, see also \cite{Devolder2013smooth}}).

Let us show that the notion of
$\psi_{\delta, L, q} (x,y)$-model can be useful for solving problems with exact first-order information but with a lower level of smoothness. Let $f$ be a subdifferentiable
function on $Q$. For each $y \in Q$, denote by $\nabla f(y)$ an arbitrary element of the subdifferential $\partial f (y)$. Assume that $f(x)$ has H\"{o}lder-continuous subgradients, i.e., it holds the following  
\begin{equation}\label{holder_cond}
    \|\nabla f(x) - \nabla f(y)\| \leq L_{\nu} \|x - y\|^{\nu}, \quad \forall x, y \in Q,
\end{equation}
where $\nu \in [0, 1]$ is the level of smoothness, and $0 < L_{\nu} <  \infty$. This condition leads to the following inequality
\begin{equation}\label{12021jh}
    f(x) \leq f(y) + \langle \nabla f(y), x - y \rangle + \frac{L_{\nu}}{1+ \nu} \|x - y\|^{1+\nu}, \quad \forall x, y \in Q. 
\end{equation}

Let us fix $\nu \in [0,1]$ and an arbitrary $\delta > 0$. The constant $L$ which depends on $\delta >0$, i.e., $L(\delta)$, such that the following inequality holds
\begin{equation}\label{ewed45s}
    \frac{L_\nu}{1+\nu}\|x-y\|^{1+\nu} \leq \frac{L(\delta)}{2}\|x-y\|^2+\delta\|x-y\|^q, 
\end{equation}
for some $q$, is the following \cite{Nabou2024qoracle}
\[
    L(\delta)=\frac{1+\nu-q}{2-q}\left(\frac{L_\nu}{1+\nu}\right)^{\frac{2-q}{1+\nu-q}}\left(\frac{1-\nu}{\delta(2-q)}\right)^{\frac{1-\nu}{1+\nu-q}}, 
\]
for any $q \in [0,1+ \nu).$

Thus, from \eqref{12021jh} and \eqref{ewed45s}, we have 
\begin{equation*}
    f(x) - \left( f(y) + \langle \nabla f(y), x - y \rangle \right) \leq  \frac{L(\delta)}{2}\|x-y\|^2+\delta\|x-y\|^q, \quad \forall x \in Q. 
\end{equation*}

Therefore, $\psi_{\delta, L, q} (x,y) =\langle \nabla f(y), x - y \rangle $ is a $(\delta, L, q)$-model of the function $f(x)$ at a given point $y \in Q$, with any $\delta >0, L = L(\delta), q \in [0, 1 + \nu)$, and $f_{\delta, L, q} (y) = f(y)$.

If we assume that the function $f(x)$ is convex on $Q$, then we have 
\begin{equation*}
    0 \leq f(x) - \left( f(y) + \langle \nabla f(y), x - y \rangle \right) \leq  \frac{L(\delta)}{2}\|x-y\|^2+\delta\|x-y\|^q. 
\end{equation*}

This means that $\psi_{\delta, L, q} (x,y) =\langle \nabla f(y), x - y \rangle $ also represents a $(\delta, L, q)$-model of the convex function $f(x)$ at a given point $y \in Q$, for any $x \in Q$,  in the sense of Definition \ref{def_neq_inex_model_convex}.

When $\nu = 1$, we get smooth functions (i.e., functions with Lipschitz continuous gradient). For $ \nu < 1$, we get a lower level of smoothness. In particular, when $ \nu = 0$, we obtain functions whose subgradients have bounded variation, and in this case, we see that the Definition \ref{def_neq_inex_model_convex} is convenient, with degree $q \in [0, 1)$, for the class of non-smooth convex optimization problems with bounded subgradient of the objective function. 
\end{example}

\section{Adaptive inexact gradient method with $(\delta, L, q)$-model}\label{sect_adaptiveGM}

In this section, we assume that the objective function $f$ is convex and $L$-smooth. For problem \eqref{main_min_prob}, with $(\delta, L,q)$-model of degree $q$ of convex function $f$, we consider an adaptive inexact gradient-type method, listed as Algorithm \ref{alg_AdapGM}.

\begin{algorithm}[H]
\caption{Adaptive inexact gradient method with $(\delta, L, q)$-model. }
\label{alg_AdapGM}
\textbf{Inputs:} $x_0 \in Q$ is the starting point s.t. $\frac{1}{2}\|x_* - x_0\|^2 \leq R^2$  for some $R>0$,  $L_0 > 0$, $\delta \geq 0$ is the oracle error, $q \in [0, 2)$ is the degree of the oracle.
\hspace*{\algorithmicindent}
\begin{algorithmic}[1]
\FOR{$k \geq 0$}
\STATE Find the smallest integer $i_k\geq 0$ such that
\begin{align}\label{exitLDL_G_S}
    f_{\delta, L, q}(x_{k+1}) & \leq f_{\delta,L, q}(x_{k}) + \psi_{\delta, L,q}(x_{k+1}, x_{k}) + \frac{L_{k+1}}{2} \left \|x_{k+1} - x_k \right\|^2 
    \\& \quad + \delta \left\|x_{k+1} - x_k \right\|^q, \nonumber
\end{align}
where $L_{k+1} = 2^{i_k-1}L_k$.
\STATE Calculate
\begin{equation}\label{equmir2DL_G_S}
    x_{k+1} := {\arg\min_{x \in Q}}\left\{\psi_{\delta , L,q}(x, x_k)  +  \frac{L_{k+1}}{2} \left\|x -  x_k\right \|^2   \right\}.
\end{equation}
\ENDFOR
\end{algorithmic}
\end{algorithm}

For Algorithm \ref{alg_AdapGM}, we have the following result. 

\begin{theorem}\label{theo_Adaptive_GM}
Assume that $\psi_{\delta,L,q}(x,y)$ is a $(\delta, L, q)$-model according to Definition \ref{def_neq_inex_model_convex}. After $N \geq 1$ iterations of Algorithm \ref{alg_AdapGM}, we have
\begin{equation}\label{rate_GM}
    f(\widehat{x}_N) - f(x_*) \leq \frac{2L R^2}{N} + \frac{2 \left(\sqrt{2} R\right)^q}{N^{q/2}} \delta,
\end{equation}
where $\widehat{x}_N = \frac{1}{\sum_{k = }^{N-1} \frac{1}{L_{k+1}}} \sum_{k = 0}^{N-1} \frac{x_{k+1}}{L_{k+1}} $.
\end{theorem}

\begin{proof}
Since $\psi_{\delta,L, q} (x, y)$ corresponds a $(\widehat{\delta}, \widehat{L})$-model of the function $f$ in a sense of \cite{Gasnikov2019fast} (see Remark \ref{rem_ordin_model}), then  from \cite{Gasnikov2019fast} for Algorithm \ref{alg_AdapGM}, we have
\begin{equation*}
    f(\widehat{x}_N) - f(x_*) \leq \frac{2 \widehat{L} R^2}{N} + 2 \widehat{\delta}, \quad \forall N \geq 1. 
\end{equation*}

By using \eqref{parameters_spec}, we get
\begin{equation}\label{eq_11}
    f(\widehat{x}_N) - f(x_*) \leq \frac{2 ( L + q \rho)}{N} R^2 + \frac{(2-q) \delta^{\frac{2}{2-q}}}{ \rho^{\frac{q}{2-q}}}, \quad \forall \rho >0.
\end{equation}

By minimizing the right hand side of \eqref{eq_11} over $\rho >0$, we find that the optimal value of $\rho$ is $\rho^* = \left(\sqrt{2}R\right)^{q-2} N^{\frac{2-q}{2}} \delta$. Therefore, we have
\begin{align*}
  f(\widehat{x}_N) - f(x_*) & \leq \frac{2LR^2}{N} + \frac{2 \rho^* q R^2}{N} +    \frac{(2-q) \delta^{\frac{2}{2-q}}}{ (\rho^*)^{\frac{q}{2-q}}}
  \\& \leq \frac{2LR^2}{N} + \frac{q \left(\sqrt{2}R\right)^q}{N^{q/2}}\delta + \frac{(2-q) \left(\sqrt{2}R\right)^q }{N^{q/2}}  \delta
  \\& = \frac{2LR^2}{N} + \frac{2 \left(\sqrt{2}R\right)^q }{N^{q/2}}\delta. 
\end{align*}
\end{proof}

\begin{remark}
From \eqref{rate_GM}, we can see that the convergence rate of Algorithm \ref{alg_AdapGM} is of order $O\left( \frac{1}{N} + \frac{\delta}{N^{q/2}} \right)$, and the second term in \eqref{rate_GM} diminishes for any $q>0$, while in \cite{Devolder2013smooth,Gasnikov2019fast,Stonyakin2021OMS} the rate is of order $O\left( \frac{1}{N} + \delta \right)$, and the second term always remains constant equals $\delta$.  
\end{remark}

\section{Adaptive inexact fast gradient method with $(\delta, L, q)$-model}\label{sec:FGM}

In this section, we consider an acceleration version of Algorithm \ref{alg_AdapGM}. Firstly, we consider the case when the objective function is smooth. For this setting of the problem, we propose  Algorithm \ref{alg2_FGM} and prove its convergence rate. Secondly, we use the technique of restarting Algorithm \ref{alg2_FGM}, when the objective function is strongly convex. Finally, under some additional conditions, we show that Algorithm \ref{alg2_FGM} is universal and applicable for solving optimization problems with a weaker level of smoothness of the objective functions. 

\subsection{Smooth convex case}

Let us assume that the objective function $f$ is convex and $L$-smooth. For solving problem \eqref{main_min_prob} with such functions, we propose an acceleration version of Algorithm \ref{alg_AdapGM}. It is listed below as Algorithm \ref{alg2_FGM}.

\begin{algorithm}
\caption{{Adaptive inexact fast  gradient method with $(\delta, L, q)$-model}. }
\label{alg2_FGM}
\textbf{Inputs:} $x_0 \in Q$ is the starting point s.t. $\frac{1}{2}\|x_0 - x_*\|^2 \leq R^2$, $ L_{0} > 0$, $\delta $ is the oracle error, $q \in [0, 2)$ is the degree of the oracle. 
\hspace*{\algorithmicindent}
\begin{algorithmic}[1]
\STATE Set
$y_0 := x_0$, $u_0 := x_0$, $\alpha_0 := 0$, $A_0 := \alpha_0$
\FOR{$k \geq 0$}
\STATE Find the smallest integer $i_k \geq 0$ such that
\begin{align}\label{exitLDL_strong}
    f_{\delta, L, q}(x_{k+1}) & \leq f_{\delta, L, q}(y_{k+1}) + \psi_{\delta, L, q}(x_{k+1}, y_{k+1}) +\frac{L_{k+1}}{2} \left\|x_{k+1} - y_{k+1}\right\|^2 
    \\& \quad + \delta\left \|x_{k+1} - y_{k+1}\right\|^q, \nonumber
\end{align}
where $L_{k+1} = 2^{i_k-1}L_k$. 

Calculate $\alpha_{k+1}$ the largest root  of the equation
\begin{gather*}\label{alpha_def_strong}
    L_{k+1}\alpha^2_{k+1}-  \alpha_{k+1} -A_k= 0,
\end{gather*}
\begin{gather}
    y_{k+1} := \frac{\alpha_{k+1}u_k + A_k x_k}{A_{k+1}}, \label{eqymir2DL_strong}
\end{gather}
\begin{equation}\label{equmir2DL_strong}
    u_{k+1} := \arg\min_{x \in Q}\left\{ \alpha_{k+1}\psi_{\delta,L, q}(x, y_{k+1}) + \frac{1}{2}\left \|x - u_k\right\|^2 \right\} , 
\end{equation}
\begin{gather}\label{eqxmir2DL_strong}
    x_{k+1} := \frac{\alpha_{k+1}u_{k+1} + A_k x_k}{A_{k+1}}. 
\end{gather}
\ENDFOR
\end{algorithmic}
\end{algorithm}

For the convergence rate of Algorithm \ref{alg2_FGM}, we have the following result. 

\begin{theorem}\label{theo_Adaptive_FGM}
Assume that $\psi_{\delta,L,q}(x,y)$ is a $(\delta, L, q)$-model according to Definition \ref{def_neq_inex_model_convex}. After $N \geq 1$ iterations of Algorithm \ref{alg2_FGM}, we have
\begin{equation}\label{rate_FGM}
    f(x_N) - f(x_*) \leq \frac{8L R^2}{(N+1)^2} + \frac{2 \left(2\sqrt{2} R\right)^q}{N^{\frac{3q}{2} - 1}} \delta. 
\end{equation}
\end{theorem}

\begin{proof}
Since $\psi_{\delta,L, q} (x, y)$ corresponds a $(\widehat{\delta}, \widehat{L})$-model of the function $f$ in a sense of \cite{Gasnikov2019fast} (see Remark \ref{rem_ordin_model}), then for Algorithm \ref{alg2_FGM}, we have \cite{Gasnikov2019fast}
\begin{equation*}
    f(x_N) - f(x_*) \leq \frac{8 \widehat{L} R^2}{(N+1)^2} + 2 N \widehat{\delta}, \quad \forall N \geq 1. 
\end{equation*} 

By using \eqref{parameters_spec}, for any $\rho > 0$, we have
\begin{align*}
    f(x_N) - f(x_*) 
    & \leq \frac{8 (L + q \rho) R^2}{(N+1)^2} + \frac{(2-q) \delta^{\frac{2}{2-q}}}{ \rho^{\frac{q}{2-q}}} N 
    \\& = \frac{8 L R^2}{(N+1)^2} + \frac{8 q R^2}{(N+1)^2} \rho + (2-q) \delta^{\frac{2}{2-q}} \rho^{\frac{q}{q-2}} N.
\end{align*}

Let $\varphi_1(\rho) : = \frac{8 q R^2}{(N+1)^2} \rho + (2-q) \delta^{\frac{2}{2-q}} \rho^{\frac{q}{q-2}} N$.  By minimizing $\varphi_1(\rho)$ over $\rho >0$, we find that the optimal value of $\rho$ is $\rho^* = \left(2\sqrt{2}R\right)^{q-2} \left(N(N+1)^2\right)^{\frac{2-q}{2}} \delta$. 

Therefore, we have
\begin{align*}
    f(x_N) - f(x_*) & \leq  \frac{8 L R^2}{(N+1)^2} + \frac{8 q R^2}{(N+1)^2} \rho^* + (2-q) \delta^{\frac{2}{2-q}} \left(\rho^*\right)^{\frac{q}{q-2}} N
    \\& = \frac{8 L R^2}{(N+1)^2} + \frac{8 q R^2}{(N+1)^2} \left(2\sqrt{2}R\right)^{q-2} \left(N(N+1)^2\right)^{\frac{2-q}{2}} \delta 
    \\& \quad + (2-q) \delta^{\frac{2}{2-q}} \left(\left(2\sqrt{2}R\right)^{q-2} \left(N(N+1)^2\right)^{\frac{2-q}{2}} \delta\right)^{\frac{q}{q-2}} N
    \\& = \frac{8 L R^2}{(N+1)^2} + \frac{q  \left(2\sqrt{2}R\right)^{q} N}{\left(N(N+1)^2\right)^{q/2}} \delta + \frac{(2-q)\left(2\sqrt{2}R\right)^{q} N}{\left(N(N+1)^2\right)^{q/2}} \delta
    \\& = \frac{8 L R^2}{(N+1)^2} + \frac{2\left(2\sqrt{2}R\right)^{q} N}{\left(N(N+1)^2\right)^{q/2}} \delta
    \\& \leq  \frac{8 L R^2}{(N+1)^2} + \frac{2 \left(2\sqrt{2}R\right)^q N}{N^{3q/2}} \delta
    \\& = \frac{8 L R^2}{(N+1)^2} + \frac{2 \left(2\sqrt{2}R\right)^q }{N^{\frac{3q}{2} - 1 }} \delta. 
\end{align*} 
\end{proof}

\begin{remark}
From \eqref{rate_FGM}, we can see that the convergence rate of Algorithm \ref{alg2_FGM} is of order $O\left( \frac{1}{N^2} + \frac{\delta}{N^{\frac{3q}{2} - 1}} \right)$. Thus for any $q > \frac{2}{3}$, we note that the error does not accumulate. While in \cite{Devolder2013smooth,Gasnikov2019fast,Stonyakin2021OMS} the rate is of order $O\left( \frac{1}{N^2} + N\delta \right)$, and the error accumulated.   
\end{remark}

\subsection{Smooth strongly convex case}\label{subsect:strongly}
Let us assume that the objective function $f$ is $\mu$-strongly convex and $L$-smooth.  

For Algorithm \ref{alg2_FGM}, we have \cite{Gasnikov2019fast}
\begin{equation*}
    f(x_N) - f(x_*) \leq \frac{4 \widehat{L} \left \|x_0 - x_* \right\|^2}{(N+1)^2} + 2 N \widehat{\delta}, \quad \forall N \geq 1. 
\end{equation*} 

By using \eqref{parameters_spec}, for any $\rho > 0$, we have
\begin{align*}
    f(x_N) - f(x_*) & \leq \frac{4 (L + q \rho) \left\|x_0 - x_* \right\|^2}{(N+1)^2} + \frac{(2-q) \delta^{\frac{2}{2-q}}}{ \rho^{\frac{q}{2-q}}} N 
    \\& = \frac{4 L \left \|x_0 - x_* \right \|^2}{(N+1)^2} + \frac{4 q \left\|x_0 - x_* \right\|^2}{(N+1)^2} \rho + (2-q) \delta^{\frac{2}{2-q}} \rho^{\frac{q}{q-2}} N.
\end{align*} 

Now, let us assume that $\left\|x_0 - x_* \right\|_2 \leq r$ for some $r>0$, then for any $\rho > 0$ we get 
\begin{equation*}
    f(x_N) - f(x_*) \leq \frac{4 L \left \|x_0 - x_* \right\|^2}{(N+1)^2} + \underbrace{\frac{4 q r^2}{(N+1)^2} \rho + (2-q) \delta^{\frac{2}{2-q}} \rho^{\frac{q}{q-2}} N}_{: = \varphi_2(\rho)}. 
\end{equation*} 

By minimizing $\varphi_2(\rho)$ over $\rho >0$, we find $\rho^* = (2r)^{q-2} \left(N(N+1)^2\right)^{\frac{2-q}{2}} \delta$. Therefore (in a similar way as in the proof of Theorem \ref{theo_Adaptive_FGM}), we have 
\begin{equation}\label{eq_120bv2gf}
    f(x_N) - f(x_*)  \leq \frac{4 L \left \|x_0 - x_* \right\|^2}{(N+1)^2} + \frac{2 (2r)^q}{N^{\frac{3q}{2} - 1}} \delta \leq \frac{4 L \left\|x_0 - x_* \right\|^2}{N^2} + \frac{2^{q+1} r^q}{N^{\frac{3q}{2} - 1}} \delta.
\end{equation}
     
For the problems with strongly convex functions (i.e.,  the objective function $f$ is $\mu$-strongly convex), we use the technique of restarting Algorithm \ref{alg2_FGM}, to accelerate its convergence rate in the following way. 

Since $f$ is $\mu$-strongly convex, we have 
\[
    f(x_N) - f(x_*) \geq \frac{\mu}{2} \left\|x_N - x_*\right\|^2, \quad \forall \mu > 0. 
\]

Thus, from \eqref{eq_120bv2gf}, we get 
\begin{equation}\label{star1}
    \left\|x_N - x_*\right\|^2 \leq \frac{8 L \left\|x_0 - x_*\right\|^2}{\mu N^2} + \frac{2^{q+2} r^q}{\mu N^{\frac{3q}{2} - 1}} \delta. 
\end{equation}

Let us take $\frac{8L}{\mu N^2} \leq \frac{1}{2}$, then $N \geq 4 \sqrt{\frac{L}{\mu}}$. Thus, after setting $N = \left\lceil 4 \sqrt{\frac{L}{\mu}} \right\rceil$, from \eqref{star1}, we find 
\begin{equation}\label{star2}
    \left\|x_N - x_*\right\|^2 \leq \frac{\left\|x_0 - x_*\right\|^2}{2} + \frac{2^{q+2} r^q}{\mu} \left\lceil\left( 4 \sqrt{\frac{L}{\mu}}\right)^{1- \frac{3q}{2}} \right\rceil \delta. 
\end{equation}

Now, let us restart Algorithm \ref{alg2_FGM} with $x_N := x^{(1)}$ as an initial point and with $N$ iterations, then we get a point $x^{(2)}$ (as an output point of the restart) for which we have 
\begin{align*}
   \left\|x^{(2)} - x_*\right\|^2 & \; \, \leq \frac{\left\|x^{(1)} - x_*\right\|^2}{2} + \frac{2^{q+2} r^q}{\mu} \left\lceil\left( 4 \sqrt{\frac{L}{\mu}}\right)^{1- \frac{3q}{2}} \right\rceil \delta
   \\& \stackrel{\eqref{star2}}{\leq} \frac{\left\|x_0 - x_*\right\|^2}{2^2} + \frac{2^{q+2} r^q}{\mu} \left\lceil\left( 4 \sqrt{\frac{L}{\mu}}\right)^{1- \frac{3q}{2}} \right\rceil \left( 1 + \frac{1}{2}\right) \delta. 
\end{align*}

Therefore, after $p - 1$ restarts of Algorithm \ref{alg2_FGM}, we get a point $x^{(p)}$ such that  
\begin{align*}
   \left\|x^{(p)} - x_*\right\|^2 & \leq \frac{\left\|x_0 - x_*\right\|^2}{2^p} + \frac{2^{q+2} r^q}{\mu} \left\lceil\left( 4 \sqrt{\frac{L}{\mu}}\right)^{1- \frac{3q}{2}} \right\rceil \left( 1 + \frac{1}{2} + \ldots + \frac{1}{2^{p-1}}\right) \delta
   \\& < \frac{\left\|x_0 - x_*\right\|^2}{2^p} + \frac{2^{q+3} r^q}{\mu} \left\lceil\left( 4 \sqrt{\frac{L}{\mu}}\right)^{1- \frac{3q}{2}} \right\rceil \delta. 
\end{align*}

Now, let us take $p \geq \log_2 \left(\frac{\left\|x_0 - x_*\right\|^2}{\varepsilon}\right) + 1 $, which implies $\frac{\left\|x_0 - x_*\right\|^2}{2^p} \leq \frac{\varepsilon}{2}$.

Also, let us choose $\delta$, such that 
\begin{equation}\label{star_3}
    \frac{2^{q+3} r^q}{\mu} \left\lceil\left( 4 \sqrt{\frac{L}{\mu}}\right)^{1- \frac{3q}{2}} \right\rceil \delta = \frac{\varepsilon}{2}. 
\end{equation}

This gives 
\begin{equation}\label{eq_I}
    \delta = \frac{\mu \varepsilon}{2^{q+4} r^q} \left\lceil\left( 4 \sqrt{\frac{L}{\mu}}\right)^{ \frac{3q}{2} - 1 } \right\rceil.  
\end{equation}

Therefore, after $p = \left \lceil \log_2 \left(\frac{\left\|x_0 - x_*\right\|^2}{\varepsilon}\right) + 1 \right\rceil $ restarts of Algorithm \ref{alg2_FGM}, with $N =  \left \lceil4 \sqrt{\frac{L}{\mu}} \right\rceil$ iterations in each restart, we get $\|x^{(p)} - x_*\|_2^2 \leq \varepsilon$, and the total number of oracle calls is 
\[
    O\left( \sqrt{\frac{L}{\mu}} \log_2 \left(\frac{\left\|x_0 - x_*\right\|^2}{\varepsilon}\right) \right) . 
\]

\begin{remark}
From \eqref{star_3}, we can find the desired accuracy $\varepsilon$ as a function of $q \in [0, 2)$, as follows
\begin{equation}\label{eq1204lkj}
    \varepsilon(q) = \frac{2^{q+4} r^q}{\mu} \left\lceil\left( 4 \sqrt{\frac{L}{\mu}}\right)^{1- \frac{3q}{2}} \right\rceil \delta.
\end{equation}

From \eqref{eq1204lkj}, we find that $\varepsilon(q) < \varepsilon(0), \forall q > \frac{2}{3}$, and thus we get a solution to the minimization problem with higher accuracy (and with a better interval for the $\delta$ in \eqref{eq_I} since $\frac{3q}{2} - 1 > 0, \forall q > \frac{2}{3}$) better than the accuracy concluded by algorithms with the inexact model of a function which considered in \cite{Gasnikov2019fast,Stonyakin2021OMS}, that is with a $(\delta, L, 0)$-model. This shows the feature of the proposed model of degree $q$ when we use this model for higher degree $q > 2/3$. 
\end{remark}

\subsection{Universal Fast Gradient Method}\label{subsect:Universal_mothod}

Let us assume that the error $\delta$, in Algorithm \ref{alg2_FGM}, can depend on the iteration counter $k$, which is indicated by input sequence $\{\delta_k\}_{k \geq 0}$. For instance, this allows obtaining the Universal Fast Gradient Method (UFGM) in which different values of  $\{\delta_k\}_{k \geq 0}$ are required (see \cite{Baimurzina2019Universal,Nesterov2015Universal}) in each iteration. 

Let us assume that the convex set $Q$ is bounded, i.e., there is $R>0$ such that for any $x,y \in Q$, we have 
\begin{equation}
    \frac{1}{2 }\|x - y\|^2 \leq R^2 \Longrightarrow \|x - y\|^q \leq \left(\sqrt{2} R\right)^q, \quad \forall q \in [0, 2). 
\end{equation}

Let us, also, assume that the function $f$ has H\"{o}lder-continuous subgradients, i.e., it holds the following inequality 
\begin{equation}\label{holder_condd}
    \left\|\nabla f(x) - \nabla f(y)\right\| \leq L_{\nu} \|x - y\|^{\nu}, \quad \forall x, y \in Q,
\end{equation}
where $\nu \in [0, 1]$ and $0 < L_{\nu} <  \infty$. 

From this, we can get the following inequality
\begin{equation*}
    f(x) - \left( f(y) + \langle \nabla f(y), x - y \rangle \right) \leq  \frac{L(\delta)}{2}\|x-y\|^2+\delta\|x-y\|^q, \quad \forall \delta > 0,  
\end{equation*}
where 
\[
    L(\delta)=\frac{1+\nu-q}{2-q}\left(\frac{L_\nu}{1+\nu}\right)^{\frac{2-q}{1+\nu-q}}\left(\frac{1-\nu}{\delta(2-q)}\right)^{\frac{1-\nu}{1+\nu-q}},
\]
for any $q \in [0,1+ \nu)$ (see Example \ref{ex4_inexact}).  

Using the same arguments as in Lemma 4 and Theorem 2 in \cite{Gasnikov2019fast}, we get the following result
\begin{equation}\label{refdf5215}
    f(x_N) - f(x_*) \leq \frac{R^2}{A_N} + \frac{2\left(\sqrt{2} R\right)^q \sum_{k = 0}^{N-1} \delta_k A_{k+1}}{A_N}. 
\end{equation}

Now, let us set $\delta_k = \frac{\alpha_{k+1} \varepsilon}{4\left(\sqrt{2} R\right)^q  A_{k+1} } $, where $\varepsilon$ is the desired accuracy of a solution. Then \eqref{refdf5215}, become in the following form
\begin{equation}\label{refdf1201}
    f(x_N) - f(x_*) \leq \frac{R^2}{A_N} + \frac{\varepsilon}{2}. 
\end{equation}

Using the same arguments as in Theorem 3 of \cite{Nesterov2015Universal}, we obtain that
$$
    A_N \geq \frac{N^{\frac{1+3 \nu}{1+\nu}} \varepsilon^{\frac{1-\nu}{1+\nu}}}{2^{\frac{2+4 \nu}{1+\nu}} L_{\nu}^{\frac{2}{1+\nu}}}.
$$
Hence, we conclude that
\begin{equation}\label{bound_Universal}
    N \leq \inf_{\nu \in[0,1]}\left[2^{\frac{3+5 \nu}{1+3 \nu}}\left(\frac{L_{\nu} R^{1+\nu}}{\varepsilon}\right)^{\frac{2}{1+3 \nu}}\right], 
\end{equation}
where the infimum can be taken since neither $\nu$ nor $L_{\nu}$ is not used in the algorithm.

The bound \eqref{bound_Universal} is optimal up to a numerical factor \cite{Guzman2015lower}.

%%%%%%%%%%%%%%%%%%%%%%%%%%%%%%%%%%%%%%%%%%%%%%%%
\section{Inexact Gradient method with $(\delta, L, \mu, q)$-oracle}\label{sec:higher_deg_oracle}

\begin{definition}\label{def:inex_model_stronglyconvex}
Let $f: Q \longrightarrow \mathbb{R}$ be a convex function, $\delta \geq 0, L>, \mu > 0$ and $q \in [0,2)$. The pair $(f_{\delta, L, \mu, q} (y), \psi_{\delta, L , \mu,   q} (x, y))$ is called a $(\delta, L, \mu, q)$-model of degree $q$ of the function $f(x)$ at the given point $y \in Q$, if it holds the following inequalities
\begin{equation}\label{ineq_inex_model_stronglyconvex}
    \frac{\mu}{2} \|x - y\|^2 \leq f(x) - \left( f_{\delta, L, \mu, q} (y) + \psi_{\delta, L, \mu, q} (x, y)\right) \leq \frac{L}{2} \|x - y\|^2 + \delta \|x - y\|^q, \quad \forall x \in Q,
\end{equation}
and $\psi_{\delta, L, \mu,  q} (x, y)$  is convex in $x$, satisfies $\psi_{\delta, L , \mu,  q}(x,x) = 0$  for all $x \in Q$. 
\end{definition}

Let us set  $\psi_{\delta, L, \mu,  q} (x, y) = \left\langle g_{\delta, L, \mu,  q}(y) , x - y\right\rangle$, where $g_{\delta, L, \mu,  q}(y) \in \textbf{E}^*$. Then we can formulate the following definition.

\begin{definition}\label{def_q_oracle_strongly}
Let $f: Q \longrightarrow \mathbb{R}$ be a convex function on a convex set $Q$. We say that $f$ is equipped with a $(\delta, L, \mu, q)$-oracle of degree $q \in [0, 2)$, with $\delta \geq 0, L>0,$ and $ \mu > 0$,  if for any $y \in Q$ we can compute a pair $\left(f_{\delta, L, \mu, q}(y), g_{\delta, L, \mu, q}(y)\right) \in \mathbb{R} \times \mathbb{R}^n$, such that 
\begin{equation}\label{def_q_oracle_srongly}
    \frac{\mu}{2} \|x - y\|^2 \leq f(x) - \left(f_{\delta, L, \mu, q}(y) + \left \langle g_{\delta, L, \mu, q} (y), x - y  \right \rangle  \right) \leq \frac{L}{2} \|x - y\|^2 + \delta \|x - y\|^q.
\end{equation}
\end{definition}

\begin{remark}
From \eqref{eq_1242} and \eqref{def_q_oracle_srongly}, we get
\begin{equation*}
   \frac{\mu}{2} \|x - y\|^2  \leq f(x) - \left(f_{\delta, L, \mu, q}(y) + \left \langle g_{\delta, L, \mu, q} (y), x - y  \right \rangle  \right) \leq \frac{\widehat{L}}{2} \|x - y\|^2 + \widehat{\delta},
\end{equation*}
where $\widehat{L} = L + q \rho$ and $\widehat{\delta} = \frac{(2-q) \delta^{\frac{2}{2-q}}}{2 \rho^{\frac{q}{2-q}}}$. 

Thus, $(\delta, L, \mu, q)$-oracle corresponds a $(\widehat{\delta}, \widehat{L}, \mu)$-oracle in the sense of \cite{Devolder2013strongly}.

Also, note that the $(\delta, L, \mu)$-oracle which considered in \cite{Devolder2013strongly} is a $(\delta, L, \mu, 0)$-oracle in the sense of the Definition \ref{def_q_oracle_strongly}.
\end{remark}

For problem \eqref{main_min_prob}, with $(\delta, L,\mu, q)$-oracle of degree $q \in [0,2)$ of convex function $f$, we consider an inexact gradient-type method, listed as Algorithm \ref{alg_GM_strongly_convex}. 

\begin{algorithm}[!ht]
\caption{Inexact gradient method with $(\delta, L, \mu, q)$-oracle.}\label{alg_GM_strongly_convex}
\textbf{Inputs:} initial point $x_0  \in Q$, $L>0$. 
\hspace*{\algorithmicindent}
\begin{algorithmic}[1]
\FOR{$k= 0, 1, \ldots $}
\STATE Obtain $\left(f_{\delta, L,\mu, q}(x_k), g_{\delta, L, \mu, q}(x_k)\right)$.
\STATE Calculate 
    \begin{equation}\label{prob_alg_3}
        x_{k+1} = \arg\min_{x \in Q } \left\{ \left\langle g_{\delta, L, \mu, q} (x_k), x- x_k \right\rangle + \frac{L}{2} \left \|x - x_k\right\|^2 \right\}.
    \end{equation}
\ENDFOR
\end{algorithmic}
\end{algorithm}

\begin{theorem}
Assume that $f$ is endowed with a $(\delta, L, \mu, q)$-oracle with $\delta \geq 0, L>0, \mu >0$ and $q \in [0, 2)$. Then for the sequence $\widehat{x}_k = \arg\min_{0 \leq i \leq k-1} f(x_i)$, generated by Algorithm \ref{alg_GM_strongly_convex}, it holds the following inequality
\begin{equation}\label{rate_strongly_lin}
    f\left(\widehat{x}_k\right) - f(x_*) \leq \frac{L r_0^2}{2} \exp \left(- k \frac{\mu}{L}\right) + \delta \sum_{ i = 0}^{k-1} \left[\left(1 - \frac{\mu}{L}\right)^i \left \|x_{k - i} - x_{k - 1 - i}\right\|^q\right],
\end{equation}
where $r_0 = \left\|x_0 - x_*\right\|$. 
\end{theorem}

\begin{proof}
Let us denote $r_k = \left\|x_k - x_*\right\|, f_k = f_{\delta, L, \mu, q} (x_k)$ and $g_k = g_{\delta, L, \mu, q} (x_k)$ for any $k \geq 0$. Then we have 
\begin{align*}
    r_{k+1}^2 & = \left \|x_{k+1} - x_k + x_k- x_*\right \|^2
    \\& = \left\|x_{k+1} - x_k\right\|^2 + 2 \left\langle  x_{k+1} - x_k, x_k - x_{k+1} + x_{k+1} - x_*\right\rangle + \left \|x_k- x_*\right\|^2
    \\& = \left\|x_{k+1} - x_k\right\|^2 - 2 \|x_{k+1} - x_k\|^2 + 2 \left\langle  x_{k+1} - x_k, x_{k+1}  - x_*\right\rangle +\left \|x_{k} - x_*\right\|^2 
    \\& = r_k^2 + 2 \left\langle  x_{k+1} - x_k, x_{k+1}  - x_*\right\rangle - \left\|x_{k+1} - x_k\right\|^2.
\end{align*}

By using the optimality condition of the problem \eqref{prob_alg_3},
\[
    \left\langle g_k + L(x_{k+1} - x_k) ,  x - x_{k+1}\right\rangle \geq 0, \quad \forall x \in Q,
\]
we find 
\[
    \left\langle x_{k+1} - x_k , x_{k+1}  - x_* \right\rangle  \leq \frac{1}{L} \left\langle g_k,  x_* - x_{k+1}\right\rangle . 
\]

Thus, we get 
\begin{align*}
    r_{k+1}^2 & \; \, \leq r_k^2 + \frac{2}{L} \left\langle g_k,  x_* - x_{k+1}\right\rangle  - \left\|x_{k+1} - x_k\right\|^2
    \\&\; \, =  r_k^2 + \frac{2}{L} \left\langle g_k,  x_* -  x_k + x_k - x_{k+1}\right\rangle  - \left\|x_{k+1} - x_k\right\|^2
    \\&\; \, = r_k^2 +  \frac{2}{L} \left\langle g_k,  x_* -  x_k \right\rangle -  \frac{2}{L} \left( \left\langle g_k,  x_{k+1} -  x_k \right\rangle  + \frac{L}{2}\left\|x_{k+1} - x_k\right\|^2  \right)
    \\& \stackrel{\eqref{def_q_oracle_srongly}}{\leq} r_k^2 +  \frac{2}{L} \left(f(x_*) - f_k - \frac{\mu}{2} \left\|x_* - x_k\right\|^2\right) 
    \\& \quad - \frac{2}{L} \left(f\left(x_{k+1}\right) - f_k - \delta \left\|x_{k+1} - x_k \right \|^q\right)
    \\&\; \, =  \left( 1 - \frac{\mu}{L}\right) r_k^2 + \frac{2}{L} \left(f(x_*) - f\left(x_{k+1}\right) \right) + \frac{2\delta }{L} \left  \|x_{k+1} - x_k\right \|^q. 
\end{align*}

By a recursive application of the last inequality, we get the following
\begin{align*}
    0 & \leq r_{k+1}^2 \\&
    \leq \left( 1 - \frac{\mu}{L}\right) \left[  \left( 1 - \frac{\mu}{L}\right) r_{k-1}^2 + \frac{2}{L} \left(f(x_*) - f\left(x_{k}\right) \right) + \frac{2\delta }{L}   \left\|x_{k} - x_{k-1} \right\|^q \right] 
    \\& \quad +\frac{2}{L} \left(f(x_*) - f\left(x_{k+1}\right) \right)  + \frac{2\delta }{L}   \left\|x_{k+1} - x_k \right\|^q 
    \\& = \left( 1 - \frac{\mu}{L}\right)^2 r_{k-1}^2 + \frac{2}{L} \left[ \left(1- \frac{\mu}{L}\right) \left(f(x_*) - f\left(x_{k}\right) \right) + f(x_*) - f\left(x_{k+1}\right)   \right] 
    \\& \quad + \frac{2\delta}{L} \left[\left(1- \frac{\mu}{L}\right) \left\|x_{k} - x_{k-1}\right \|^q + \left\|x_{k+1} - x_k\right \|^q  \right]
     \leq \ldots
     \\& \leq  \left( 1 - \frac{\mu}{L}\right)^{k+1} r_{0}^2 + \frac{2}{L} \sum_{i = 0}^{k}  \left[\left(1- \frac{\mu}{L}\right)^i \left(f(x_*) - f\left(x_{k+1-i}\right) \right) \right]
     \\& \quad + \frac{2\delta}{L} \sum_{i = 0}^{k} \left(1- \frac{\mu}{L}\right)^i\left \|x_{k+1-i} - x_{k-i}\right \|^q .
\end{align*}

Therefore, we have
\begin{align*}
    & \quad \; \frac{2}{L} \sum_{i = 0}^{k} \left[ \left(1- \frac{\mu}{L}\right)^i\left( f\left(x_{k+1-i}\right) - f(x_*) \right)   \right] 
    \\& \leq \left( 1 - \frac{\mu}{L}\right)^{k+1} r_{0}^2+  \frac{2\delta}{L} \sum_{i = 0}^{k} \left(1- \frac{\mu}{L}\right)^i \left\|x_{k+1-i} - x_{k-i}\right \|^q .  
\end{align*}

Now, from the definition $\widehat{x}_{k+1} := \arg\min_{0 \leq i \leq k} f(x_i)$, and since
$$
\frac{2}{L} \sum_{i = 0}^{k} \left(1- \frac{\mu}{L}\right)^i = \frac{2}{\mu} \left[ 1 -\left(1- \frac{\mu}{L}\right)^{k+1} \right],
$$
we get the following 
\begin{align*}
    & \quad \; \left(\frac{2}{L} \sum_{i = 0}^{k} \left(1- \frac{\mu}{L}\right)^i\right) \left(f\left(\widehat{x}_{k+1}\right) - f(x_*)\right) 
    \\& \leq  \left( 1 - \frac{\mu}{L}\right)^{k+1} r_{0}^2 +  \frac{2\delta}{L} \sum_{i = 0}^{k} \left(1- \frac{\mu}{L}\right)^i \left\|x_{k+1-i} - x_{k-i}\right \|^q .
\end{align*}

Thus, we have 
\begin{align}
    f\left(\widehat{x}_{k+1}\right) & - f(x_*) \nonumber
    \leq   \frac{\left( 1 - \frac{\mu}{L}\right)^{k+1} r_{0}^2 }{\frac{2}{\mu} \left[ 1 -\left(1- \frac{\mu}{L}\right)^{k+1} \right]}
    +  \delta \frac{ \sum_{i = 0}^{k} \left(1- \frac{\mu}{L}\right)^i \left\|x_{k+1-i} - x_{k-i}\right \|^q }{ \sum_{i = 0}^{k} \left(1- \frac{\mu}{L}\right)^i} \nonumber
    \\& \leq \frac{L}{2}  \left( 1 - \frac{\mu}{L}\right)^{k+1} r_{0}^2 
    +  \delta \frac{ \sum_{i = 0}^{k} \left(1- \frac{\mu}{L}\right)^i \left\|x_{k+1-i} - x_{k-i}\right \|^q }{ \sum_{i = 0}^{k} \left(1- \frac{\mu}{L}\right)^i} \nonumber
    \\& \leq \frac{Lr_0^2}{2} \exp\left(- (k+1) \frac{\mu}{L}\right) +  \delta \frac{ \sum_{i = 0}^{k} \left(1- \frac{\mu}{L}\right)^i \left\|x_{k+1-i} - x_{k-i}\right \|^q }{ \sum_{i = 0}^{k} \left(1- \frac{\mu}{L}\right)^i}
    \label{eq_erty}
    \\& \leq \frac{L r_0^2}{2} \exp \left(- (k+1) \frac{\mu}{L}\right) + \delta \sum_{ i = 0}^{k} \left[\left(1 - \frac{\mu}{L}\right)^i \left \|x_{k+1 - i} - x_{k - i}\right\|^q\right]. \nonumber
\end{align}
\end{proof}

\begin{remark}
The case when the objective function is smooth and strongly convex was studied in Section \ref{sec:FGM} (see Subsect. \ref{subsect:strongly}), where we used the technique of restarting Algorithm \ref{alg2_FGM}. For this technique, we consider the Euclidean setting of the considered problem and in addition to running the algorithm, we must know the function's strongly convex parameter $\mu$.  Whilst in Algorithm \ref{alg_GM_strongly_convex}, there are no restrictions on using of any setting (i.e., any norm) of the problem. Also, there isn't a necessity to know the parameter $\mu$, in addition to the ease of implementing the Algorithm \ref{alg_GM_strongly_convex}. 
\end{remark}

\section{Inexact $(\delta, L, q)$-model for variational inequalities}\label{Sect:VI}

In this section, we consider the problem of finding a solution $x_* \in Q$ for variational inequality (VI) in the following abstract form \cite{Stonyakin2021OMS}
\begin{equation}\label{VI_prob_1}
    \psi(x, x_*) \geq 0, \quad \forall x \in Q,
\end{equation}
for some convex compact set $Q \subset \mathbb{R}^n$ and some function $\psi: Q \times Q \longrightarrow \mathbb{R}$. 

By assuming the abstract monotonicity of the function $\psi$
\begin{equation}\label{abstract_monot_VI}
    \psi(x,y) + \psi(y, x) \leq 0, \quad \forall x, y \in Q,
\end{equation}
we find that any solution to \eqref{abstract_monot_VI} is a solution to the following inequality
\begin{equation}\label{VI_max}
    \max_{x \in Q} \psi (x_*, x) \leq 0. 
\end{equation}

In the general case, we assume the existence of a solution $x_*$ of problem \eqref{VI_prob_1}. In a particular case, if for some operator $g: Q \longrightarrow \mathbb{R}^n$ we set $\psi(x, y) = \langle g(y), x - y \rangle \; \forall x, y \in Q$, then \eqref{VI_prob_1} and \eqref{VI_max} are equivalent to a standard strong and weak VI with the operator $g$, respectively \cite{Stonyakin2021OMS}.

\begin{definition}\label{Def_Model_VI}
Let $\delta>0, L > 0$ and $q \in [0, 2)$. We say that a function $\psi$ has a $(\delta, L, q)$-model $\psi_{\delta, L,q} (x, y)$ of degree $q$ for variational inequalities if the following properties hold for each $x, y, z \in Q$:
\begin{enumerate}
    \item[(i)] $\psi(x, y) \leq \psi_{\delta, L, q}(x, y) + \delta \|x-y\|^q$,
    \item[(ii)] $\psi_{\delta, L, q} (x, y)$ convex in the first variable,
    \item[(iii)] $\psi_{\delta, L, q}(x,x)=0$,
    \item[(iv)] ({\it abstract $(\delta, q)$-monotonicity})
    \begin{equation}\label{eq:abstr_monot_delta}
        \psi_{\delta, L, q}(x,y)+\psi_{\delta, L, q}(y,x)\leq \delta \|x- y \|^q,
    \end{equation}
    \item[(v)] %({\it generalized relative smoothness})
    \begin{align}\label{VIeq20}
        \psi_{\delta, L, q}(x,y)  & \leq \psi_{\delta, L, q}(x,z)+\psi_{\delta, L, q}(z,y)+ \frac{L}{2}\left(\|z - x\|^2 + \|z- y \|^2 \right) 
        \\& \quad + \frac{\delta}{2} \left(\|z - x\|^q + \|z- y \|^q\right). \nonumber
    \end{align}
\end{enumerate}
\end{definition}

\begin{remark}
From \eqref{eq_1242} and \eqref{VIeq20}, we get the following inequality
\[
    \psi_{\delta, L, q}(x,y)   \leq \psi_{\delta, L, q}(x,z)+\psi_{\delta, L, q}(z,y)+ \frac{\widehat{L}}{2}\left(\|z - x\|^2 + \|z- y \|^2 \right) + \widehat{\delta},
\]
where  $\widehat{L} := L + q \rho$ and $\widehat{\delta} := \frac{(2-q) \delta^{\frac{2}{2-q}}}{2\rho^{\frac{q}{2-q}}}$. 

Thus, $\psi_{\delta,L, q} (x, y)$ corresponds a $(\widehat{\delta}, \widehat{L})$-model for VI in a sense of \cite{Stonyakin2021OMS}. 

Note that the inexact model for variational inequalities which was considered in \cite{Stonyakin2021OMS} is a $(\delta, L, 0)$-model in a sense of the Definition \ref{Def_Model_VI}.
\end{remark}

For the abstract variational inequalities problem \eqref{VI_prob_1} (or \eqref{VI_max}),  with $(\delta, L, q)$-model, we consider an adaptive Algorithm, listed as Algorithm \ref{Alg:UMPModel}. 

\begin{algorithm}[htp]
\caption{Generalized Mirror Prox for VIs with $(\delta, L, q)$-model.}
\label{Alg:UMPModel}
\textbf{Inputs:} accuracy $\varepsilon > 0$, oracle error $\delta >0$, $q \in [0, 2)$, initial guess $L_{0} >0$. 
\begin{algorithmic}[1]
\STATE Set $k=0$, $z_0 \in Q$ such that $\max_{x \in Q} \left\|x - z_0\right\|^2 \leq D$, for some $D>0$.
\STATE \textbf{repeat}
\STATE \quad Find the smallest integer $i_k \geq 0$ such that
\begin{equation}\label{eqUMP23}
\begin{aligned}
    \psi_{\delta, L, q}(z_{k+1}, z_{k}) & \leq \psi_{\delta, L, q}(z_{k+1}, w_{k})+\psi_{\delta, L, q}(w_k,z_k) 
    \\& \quad  + \frac{L_{k+1}}{2} \left( \left\|w_k - z_k\right\|^2 + \left\|w_k - z_{k+1}\right\|^2 \right) 
    \\& \quad  +\frac{ \delta}{2} \left(\left\|w_k - z_k\right\|^q + \left\|w_k - z_{k+1}\right\|^q\right),
\end{aligned}
\end{equation}
\quad where 	$L_{k+1}=2^{i_k-1}L_{k}$, and
\begin{align}
    w_k&={\mathop {\arg\min }\limits_{x\in Q}} \left\{\psi_{\delta, L, q}(x, z_k)+ \frac{L_{k+1} }{2}\left\|x -  z_k\right\|^2  \right\}, \label{eq:UMPwStepMod} \\
    z_{k+1}&={\mathop {\arg\min }\limits_{x\in Q}} \left\{\psi_{\delta, L, q}(x, w_k) + \frac{L_{k+1} }{2} \left\|x -  z_k\right\|^2 \right\}, \label{eq:UMPzStepMod}
\end{align}
\STATE \textbf{until}
\begin{equation}\label{eq_Alg3}
    S_N:= \sum_{k=0}^{N-1}\frac{1}{L_{k+1}}\geq \frac{\max_{x \in Q} \left\|x -  z_0\right\|^2}{2 \varepsilon}.
\end{equation}
\end{algorithmic}
\textbf{Output:} $\widehat{w}_N = \frac{1}{S_N}\sum_{k=0}^{N-1}\frac{1}{L_{k+1}}w_k$.
\end{algorithm}

For the convergence rate of  Algorithm \ref{Alg:UMPModel}, we have the following result.

\begin{theorem}\label{thmm1inexact}
For Algorithm \ref{Alg:UMPModel}, it holds 
\begin{equation}\label{rate_VIs}
    \max\limits_{u \in Q}\psi(\widehat{w}_N,u) \leq \frac{LD}{N} + \left(\sqrt{\frac{2D}{3}}\right)^q \frac{3}{N^{q/2}}  \delta, \quad \forall N \geq 1.
\end{equation}
\end{theorem}

\begin{proof}
Since $\psi_{\delta, L, q}(x,y)$ (which is defined in Definition \ref{Def_Model_VI}) corresponds a $(\widehat{\delta}, \widehat{L})$-model for variational inequalities in the sense of \cite{Stonyakin2021OMS}, then for Algorithm \ref{Alg:UMPModel}, we have
\[
    \max\limits_{u \in Q}\psi(\widehat{w}_N,u) \leq \frac{\widehat{L} \max_{u \in Q} \left\|u -z_0\right\|^2}{N} + 3 \widehat{\delta} \leq  \frac{ \widehat{L} D}{N}  + 3 \widehat{\delta},
\]
where $\widehat{L} = L + q \rho$ and $ \widehat{\delta} = \frac{(2-q) \delta^{\frac{2}{2-q}}}{2 \rho^{\frac{q}{2-q}}}.$ Thus, we have
\[
    \max\limits_{u \in Q}\psi(\widehat{w}_N,u) \leq \frac{LD}{N} + \underbrace{\frac{q D}{N}\rho + \frac{3}{2} (2-q) \delta^{\frac{2}{2-q}} \rho^{\frac{q}{q-2}}}_{:=\varphi_3(\rho)}, \quad \forall \rho >0. 
\]

By minimizing $\varphi_3(\rho)$ over $\rho >0$, we find the optimal value of $\rho$ is  $\rho^* = \left(\frac{3}{2D}\right)^{\frac{2-q}{2}} \frac{\delta}{N^{(q-2)/2}}$. Therefore, we have
 
\begin{align*}
    \max\limits_{u \in Q}\psi(\widehat{w}_N,u)
    & \leq \frac{LD}{N} + \frac{q D}{N}\rho^* + \frac{3}{2} (2-q) \delta^{\frac{2}{2-q}} (\rho^*)^{\frac{q}{q-2}}
    \\& = \frac{LD}{N} + \frac{2^{\frac{q-2}{2}}  3^{\frac{2-q}{2}} D^{q/2} q}{N^{q/2}  }  \delta + 3 (2-q) \left(2^{\frac{q-2}{2}}\right)  \left (\frac{D}{3N}\right)^{q/2}   \delta
    \\& = \frac{LD}{N} + \frac{3q}{2}\left (\frac{2D}{3}\right)^{q/2}   \frac{\delta}{N^{q/2}} +  \frac{3(2-q)}{2}  \left (\frac{2D}{3}\right)^{q/2}   \frac{\delta}{N^{q/2}}
    \\& = \frac{LD}{N} + \left(\sqrt{\frac{2D}{3}}\right)^q \frac{3}{N^{q/2}}  \delta. 
\end{align*}
\end{proof}

\begin{remark}
From \eqref{rate_VIs}, we can see that the convergence rate of Algorithm \ref{Alg:UMPModel} is of order $O\left( \frac{1}{N} + \frac{\delta}{N^{q/2}} \right)$, and the second term on the right-hand side of \eqref{rate_VIs} diminishes for any $q>0$, while in \cite{Stonyakin2021OMS} the rate is of order $O\left( \frac{1}{N} + \delta \right)$, and the second term remains constant equals $\delta$.   
\end{remark}

\section{Inexact $(\delta, L, q)$-model for saddle point problems}\label{Sect:SPP}

In this section, we introduce a higher degree inexact model for saddle point problems. It is known that the solution of variational inequalities reduces the so-called saddle-point problems, in which for a convex in $u$ and concave in $v$ function $f(u,v):\mathbb{R}^{n_1+n_2}\longrightarrow\mathbb{R}$ ($u\in Q_1\subset\mathbb{R}^{n_1}$ and $v\in Q_2\subset\mathbb{R}^{n_2}$, where $Q_1$ and $Q_2$ are convex sets) needs to be found the point $(u_*, v_*)$ such that
\begin{equation}\label{eq31}
f(u_*,v)\leq f(u_*,v_*)\leq f(u,v_*), \quad \forall u\in Q_1 \quad \text{and} \quad  v\in Q_2. 
\end{equation}

Let $Q=Q_1\times Q_2\subset\mathbb{R}^{n_1+n_2}$. For $x=(u,v)\in Q$, we assume that $\|x\|=\sqrt{\|u\|^2+\|v\|^2}$.

In other words, in the saddle point problem we need to find a solution to the min-max problem 
\begin{equation}\label{SPP_problem}
    \min_{u \in Q_1} \max_{v \in Q_2 } f(u, v). 
\end{equation}

Let us denote $x=(u_x,v_x) \in Q,\;y=(u_y,v_y)\in Q$, where $Q=Q_1\times Q_2\subset\mathbb{R}^{n_1+n_2}$. It is well known that for a sufficiently smooth function $f$ with respect to $u$ and $v$, the problem \eqref{SPP_problem} reduces to a variational inequality with the following operator 
\begin{equation}\label{operator_spp}
    G(x)= \left[ \nabla_u f(u_x,v_x), \; - \nabla_v f(u_x,v_x)\right]^\top.
\end{equation}

Therefore, we can use Algorithm \ref{Alg:UMPModel}, to solve the class of saddle point problems with operator \eqref{operator_spp}.

For saddle point problems, we propose some adaptation of the concept of the $(\delta, L, q)$-model for abstract variational inequality.

\begin{definition}\label{DefSaddleModel}
We say that the function $\psi_{\delta, L,q}(x,y)$ $(\psi_{\delta, L, q}:\mathbb{R}^{n_1+n_2}\times\mathbb{R}^{n_1 + n_2}\longrightarrow\mathbb{R})$ is a $(\delta,L,q)$-model for the saddle-point problem \eqref{SPP_problem} if the conditions (ii) -- (v) of Definition \ref{Def_Model_VI} hold and in addition
\begin{equation}\label{eq33}
    f(u_y,v_x)-f(u_x,v_y)\leq-\psi_{\delta, L,q}(x,y) + \delta\|x- y\|^q,  \quad \forall x, y \in Q.
\end{equation}
\end{definition}

From Theorem \ref{thmm1inexact}, we conclude the following result for the saddle point problems.
\begin{theorem}
Let $\delta \geq 0, L >0,$ and $ q \in [0,2)$. If for the saddle-point problem \eqref{SPP_problem} there is a $(\delta,L,q)$-model $\psi_{\delta, L,q}(x,y)$, then after stopping Algorithm \ref{Alg:UMPModel},  we get a point
\begin{equation}\label{eq36}
    \widehat{y}_N=(u_{\widehat{y}_N},v_{\widehat{y}_N}):=(\widehat{u}_N, \widehat{v}_N):=\frac{1}{S_N}\sum_{k=0}^{N-1}\frac{w_{k}}{L_{k+1}},
\end{equation}
for which it holds the following inequality 
\begin{equation}\label{eq37}
    \max_{v\in Q_2}f(\widehat{u}_N, v)-\min_{u\in Q_1}f(u, \widehat{v}_N)\leq \frac{LD}{N} + \left(\sqrt{\frac{2D}{3}}\right)^q \frac{3}{N^{q/2}}  \delta.  
\end{equation}
where $D >0 $, satisfies
\[
    \max_{(u,v) \in Q} \left\|(u, v) - (u_0, v_0)\right\|^2   \leq D. 
\]
\end{theorem}

\section{Numerical experiments}\label{sect_numerical}

In Subsect. \ref{subsect:Universal_mothod}, we showed that Algorithm \ref{alg2_FGM} with a suitable choice of the time-varying sequence $\{\delta_k\}_{k \geq 0}$ is a universal algorithm (see also Example \ref{ex4_inexact}). Therefore we can use Algorithm \ref{alg2_FGM} with a time-varying sequence $\{\delta_k\}_{k \geq 0}$ for the class of non-smooth optimization problems (i.e., when $\nu = 0$), and the convenient interval of the degree $q$ of the proposed inexact model will be $[0, 1)$. From Theorem \ref{theo_Adaptive_FGM}, we can see that for any $q \in (2/3, 1)$, the error of the inexact model does not accumulate. 

To show the advantages and effects of the proposed inexact model with a universal method, a series of numerical experiments were performed for some non-smooth optimization problems with a geometrical nature. 

We compare the performance of the Universal Fast Gradient Method (UFGM), i.e., Algorithm \ref{alg2_FGM}, with a special form of the model $\psi_{\delta, L, q}(x,y) = \langle \nabla f(y), x - y \rangle$ and projected subgradient method using different famous step size rules that are listed in Table \ref{Tab_steps}. 

In our experiments, we take the set $Q$ as the unit ball in $\mathbb{R}^n$ with the center at $0 \in \mathbb{R}^n$. All compared methods start from the same initial point $x_0 = \left(\frac{1}{\sqrt{n}}, \ldots, \frac{1}{\sqrt{n}}\right) \in Q \subset \mathbb{R}^n$. In the AdaGrad algorithm, we take $\alpha = 10^{-8}$, and there is an assumption that $\|x_0 - x_*\|^2 \leq 2 \theta_0^2$, thus for the taken feasible set $Q$ in our experiments, we can take $\theta_0 = 1/\sqrt{2}$.  

The comparison of the methods is done in terms of the difference $\hat{f}_k - f_{\text{min}}$, where $\hat{f}_k$ denotes the value of the objective function $f$ at the averaged points (namely at $\hat{x}_k = \frac{1}{k} \sum_{j = 1}^{k} x^k$, for all cases in Table \ref{Tab_steps}, except to the case of ''quad grad'', where in this case we have $\hat{x}_k = \left(\sum_{j = 1}^{k} \gamma_k\right)^{-1} \sum_{j=1}^{k} \gamma_k x^k $, and except to the case of ''AdaMirror'' (adaptive mirror descent method with weighting scheme), where in this case we have $\hat{x}_k = \frac{1}{\sum_{j= 1}^{k} \gamma_j^{-m}} \sum_{j = 1}^{k} \gamma_j^{-m} x_j $ with $m \geq -1$, see \cite{moh2024asmaa} for more details about AdaMirror) and $f_{\text{min}}$ denotes a minimal value of the objective function computed by SciPy, a package for solving many different classes of optimization problems (when the dimension of the space $\mathbb{R}^n$ is not big).  

\begin{table}[htp]
\begin{tabular}{l l l }
\hline
Abbreviation & Step sizes  & Formula of $\gamma_k$ \\ \hline\hline
constant step & constant step size \cite{Boyd2004Subgradient} & $\gamma_k = 0.1,$ \\ \hline
fixed length  & fixed step length \cite{Boyd2004Subgradient} & $\gamma_k = \frac{0.2}{\|\nabla f(x^k)\|_*},$ \\ \hline
nonsum        &
\shortstack{non-summable \\ diminishing step \cite{Boyd2004Subgradient}}    & $\gamma_k = \frac{0.1}{\sqrt{k}},$ \\ \hline
sqrsum nonsum & 
\shortstack{square summable but \\ not summable step \cite{Boyd2004Subgradient}}
  & $\gamma_k = \frac{0.5}{k},$ \\ \hline
quad grad     & \shortstack{quadratic of the norm \\ of the gradient \cite{FedorRolandbook}}  & $\gamma_k = \frac{0.2}{\|\nabla f(x^k)\|_*^2},$ \\ \hline
AdaGrad       & \shortstack{AdaGrad \\ algorithm \cite{Duchi2011Adaptive}}   & $\gamma_k =  \frac{\theta_0}{\sqrt{\sum_{j = 1}^{k} \|\nabla f(x^j)\|_*^2 + \alpha}}, $ \\ \hline
Polyak step   & Polyak step size \cite{Polyak1987book} & $\gamma_k = \frac{f(x^k) - f^*}{\|\nabla f(x^k)\|_*^2},$  \\ \hline
AdaMirror & \shortstack{Adaptive step size \cite{moh2024asmaa} \\
with weighting scheme} & $\gamma_k = \frac{\sqrt{2}}{\|\nabla f(x^k)\|_* \sqrt{k}}.$ \\ \hline\hline
\end{tabular}
\caption{The used step sizes in the projected subgradient method.}
\label{Tab_steps}
\end{table}

\subsection{Best approximation problem}

The considered problem in this subsection is connected with the problem of the best approximation of the distance between a point and a given set $Q\subset \mathbb{R}^n $. For this problem, let $A \notin  Q$ be a given point, we need to solve the following optimization problem
\begin{equation}\label{best_approx_prob}
    \min_{x \in Q}\left\{ f(x) : = \left\|x - A\right\| \right\}. 
\end{equation}

The point $A$ is randomly generated from a uniform distribution over $[0, 1)$, such that $\|A\|_2 = 10$, therefore the distance between the point $A$ and the considered unit ball $Q$ is equal to $9$, i.e., $f^* = 9$. Here, we mention that this problem is constructed to use the Polyak step size, which requires knowing the optimal value $f^*$. 

The comparison results, for problem \eqref{best_approx_prob} with $n = 5000$  are presented in Fig. \ref{fig_best_approx_n1000}. In this figure, $\hat{f}_k$  denotes the value of the objective function $f$ at the averaged points in each iteration of all compared algorithms. 

From Fig. \ref{fig_best_approx_n1000}, we can see that UFGM is the best, where the difference between its performance and the rest of the algorithms with steps in Table \ref{Tab_steps} is clear and significant.

\begin{figure}[htp]
\minipage{0.95\textwidth}
\includegraphics[width=\linewidth]{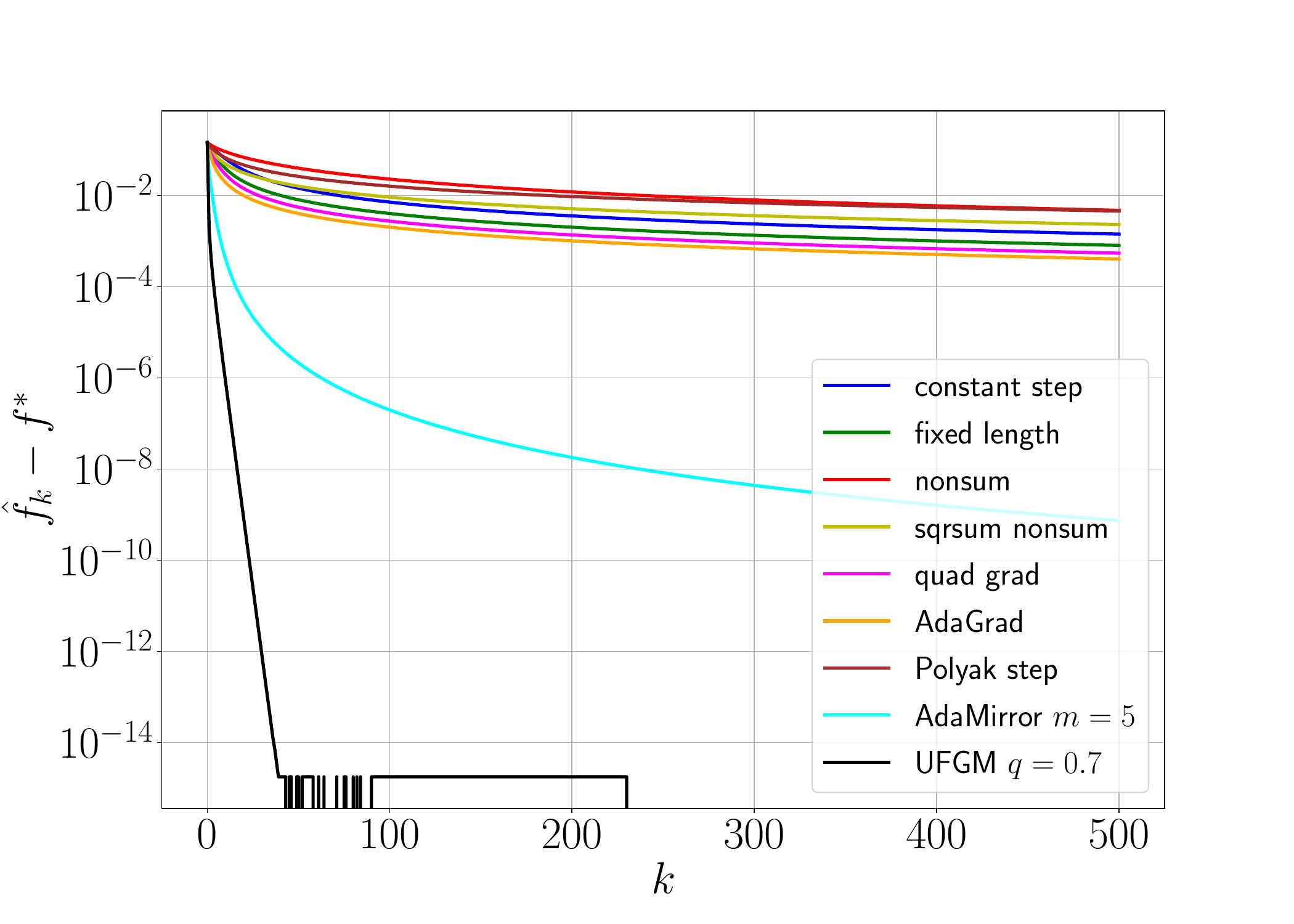}
\endminipage%\hfill
\caption{Results of UFGM and projected subgradient method using different step size rules listed in Table \ref{Tab_steps}, for problem  \eqref{best_approx_prob} with $n=5000$.}
\label{fig_best_approx_n1000}
\end{figure}

\subsection{Fermat–Torricelli–Steiner problem.}

Let $A_j \in \mathbb{R}^n, j = 1, \ldots, T$ be a given set of $T$ points, and let us consider an analogue of the well-known Fermat–Torricelli–Steiner problem. For this we need to solve the following optimization problem
\begin{equation}\label{Fermat_prob}
    \min_{x \in Q}\left\{ f(x) : = \frac{1}{T}\sum_{j = 1}^{T} \left\|x -  A_j \right \| \right\}. 
\end{equation}

The points $A_j, j = 1, 2, \ldots, T$ are randomly generated from a uniform distribution over $[0, 1)$. We run all algorithms (except the algorithm with Polyak step size since we cannot know the optimal value $f^*$ for problem \eqref{Fermat_prob}), with the same initial point $x_0 = \left(\frac{1}{\sqrt{n}}, \ldots, \frac{1}{\sqrt{n}}\right) \in Q\subset \mathbb{R}^n$. 

The comparison results, for problem \eqref{Fermat_prob} with $n = 200$ and $T= 25$,   are presented in Fig. \ref{fig_Fermat_n200}. The reason here for taking $n = 200, T = 25$ (i.e., not so big) is that we calculated the value $f_{\text{min}}$ using SciPy which does not work well for large values of $n$ and $T$. In Fig. \ref{fig_Fermat_n200}, $\hat{f}_k$  denotes the value of the objective function $f$ at the averaged points in each iteration of all compared algorithms. From this figure, we can see that UFGM is the best. Note that the difference between its performance and the rest of the algorithms (except the adaptive mirror descent method with weighting scheme (AdaMirror), where the difference is clear at the first iterations, and then after a determined number of iterations there is no difference between their efficiency) is obvious.

\begin{figure}[htp]
\minipage{0.95\textwidth}
\includegraphics[width=\linewidth]{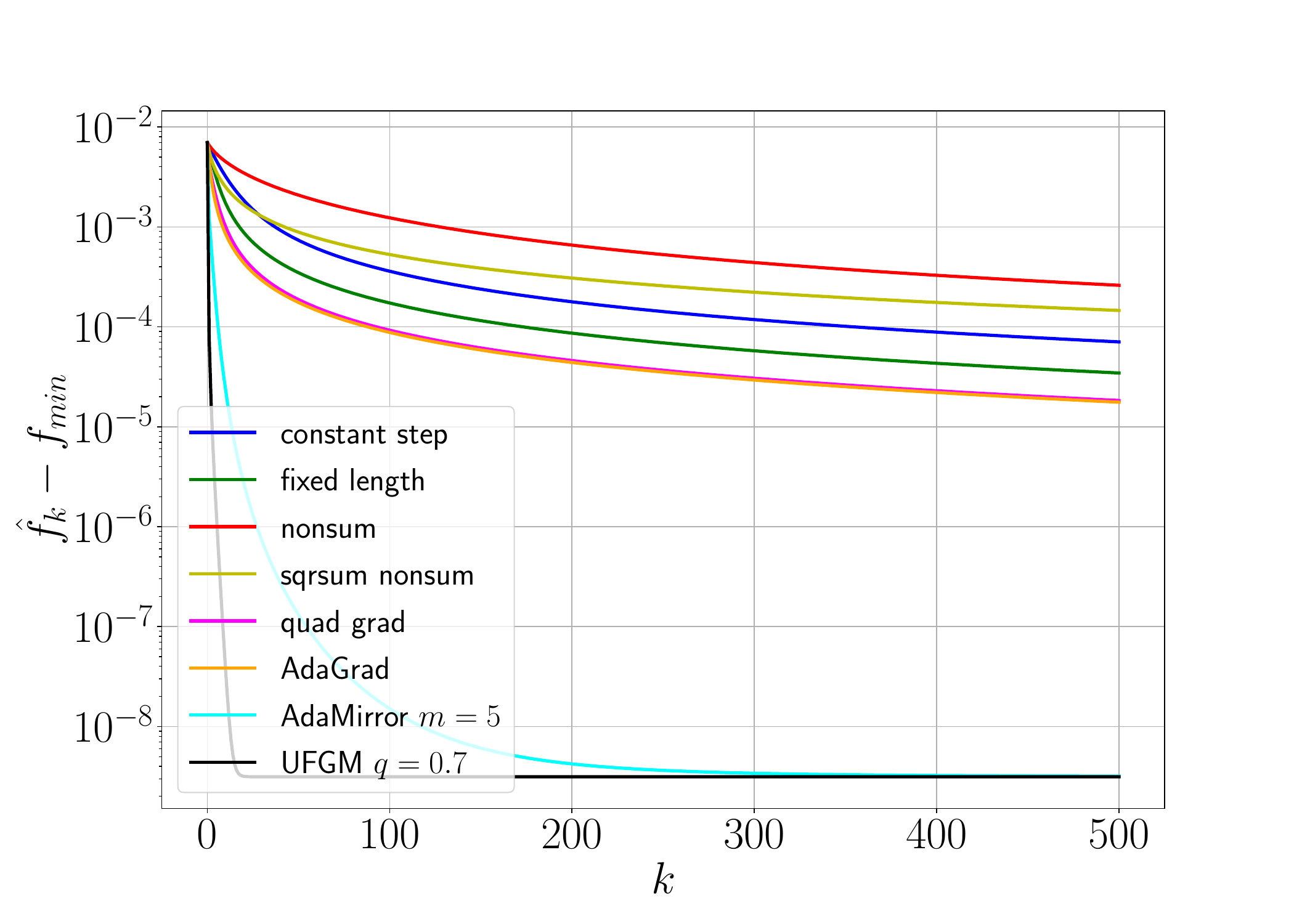}
\endminipage%\hfill
\caption{Results of  UFGM and projected subgradient method using different step size rules listed in Table \ref{Tab_steps}, for problem \eqref{Fermat_prob} with $n=200, T = 25$.}
\label{fig_Fermat_n200}
\end{figure}

\section{Conclusion}\label{sect:Conclusion}
In this paper, we introduce an inexact higher degree $(\delta, L, q)$-model for convex and non-convex optimization problems.  We obtain convergence rates for an adaptive inexact gradient method $O\left(\frac{1}{k} + \frac{\delta}{k^{q/2}}\right)$ and an adaptive inexact fast gradient method (FGM) $O\left(\frac{1}{k^2} + \frac{\delta}{k^{(3q-2)/2}}\right)$ for convex optimization problems with this model.  For the gradient method, the coefficient of $\delta$ diminishes with $k$, and for the fast gradient method, there is no error accumulation for $q \geq 2/3$. By using the technique of restarting the FGM with the proposed higher degree inexact model we obtain a convergence rate of the restarted method for strongly convex problems. Also, using the FGM we construct a universal fast gradient method (UFGM), this method allows us to solve optimization problems with different levels of smoothness including non-smooth problems. For the strongly convex optimization problems, we introduce an inexact higher degree $(\delta, L, \mu, q)$-oracle and obtain a convergence rate of the inexact gradient method with this oracle without using the technique of restarting any other algorithms. In addition to the minimization problems, we introduce an inexact higher degree $(\delta, L, q)$-model for variational inequalities and saddle-point problems and obtain convergence rate $O\left(\frac{1}{N} + \frac{\delta}{k^{q/2}}\right)$ for adaptive versions of the Generalized Mirror–Prox algorithm for problems with this model, which improves the well-known rate $O\left(\frac{1}{k} + \delta\right)$ for any $q > 0 $, where there is no error accumulation for $q > 0$. At the last, we perform some numerical experiments for testing the proposed UFGM with the proposed inexact model for non-smooth optimization problems with a geometrical nature, such as the best approximation problem and the Fermat-Torricelli-Steiner problem.

As a future work, we plan to study the accelerated and non-accelerated stochastic gradient descent methods, as in \cite{Dvinskikh2020Accelerated}, for solving finite sum optimization problems (see Example \ref{ex1_inexact}), and deriving estimates of the rate of convergence in the proposed higher degree inexact model. 

\subsection*{Acknowledgements}
This work was supported by a grant for research centers in the field of artificial intelligence, provided by the Analytical Center for the Government of the Russian Federation in accordance with the subsidy agreement (agreement identifier 000000D730324P540002) and the agreement with the Moscow Institute of Physics and Technology dated November 1, 2021 No. 70-2021-00138.


\begin{thebibliography}{99}
\bibitem{moh2024asmaa}
M. Alkousa, F. Stonyakin, A. Abdo, M. Alcheikh: Optimal Convergence Rate for Mirror Descent Methods  with special Time-Varying Step Sizes Rules. \url{https://arxiv.org/abs/2401.04754}

\bibitem{Agafonov2023}
A. Agafonov, D. Kamzolov, P. Dvurechensky, A. Gasnikov, M. Takáč: Inexact tensor methods and their application to stochastic convex optimization. 2023.  \url{https://doi.org/10.1080/10556788.2023.2261604}

\bibitem{Arora2019Fine}
S. Arora, S. Du, W. Hu, Z. Li, and R. Wang: Fine-grained analysis of optimization and generalization for overparameterized two-layer neural networks. In Kamalika Chaudhuri and Ruslan Salakhutdinov, editors, Proceedings of the 36th International Conference on Machine Learning, volume 97 of Proceedings of Machine Learning Research, pages 322–332. PMLR, 09–15 Jun 2019. \url{http://proceedings.mlr.press/v97/arora19a.html}

\bibitem{Baimurzina2019Universal}
D. Baimurzina, A. Gasnikov, E. Gasnikova, P. Dvurechensky, E. Ershov, M. Kubentaeva, and A. Lagunovskaya: Universal similar triangulars method for searching equilibriums in traffic flow distribution models, J. Comput.Math. Math. Phys. 59 (2019), pp. 21–36.

\bibitem{bauschke2016descent}
H. H. Bauschke, J. Bolte, M. Teboulle: A descent lemma beyond lipschitz gradient continuity: first-order methods revisited and applications, Math. Oper. Res. 42 (2016), pp. 330–348.

\bibitem{Beck2017}
A. Beck: First-order methods in optimization, vol. 25, SIAM, 2017.

\bibitem{Brutzkus2018SGD}
A. Brutzkus, A. Globerson, E. Malach, Sh. Shalev-Shwartz: SGD learns over-parameterized networks that provably generalize on linearly separable data. In International Conference on Learning Representations, 2018.  \url{https://openreview.net/forum?id=rJ33wwxRb}.

\bibitem{Berahas2021theoretical}
A. S. Berahas, L. Cao, K. Choromanski, K. Scheinberg: A theoretical and empirical comparison of gradient approximations in derivative-free optimization, Foundations of Computational Mathematics (2021), pp. 1-54.

\bibitem{Bottou2018machine}
L. Bottou, F. E. Curtis, J. Nocedal: Optimization methods for large-scale machine learning. SIAM Review, 60(2): 223–311, 2018. 

\bibitem{Boyd2004Subgradient}
S. Boyd, L. Xiao, A. Mutapcic: Subgradient methods. lecture notes of EE392o, Stanford University, Autumn Quarter, 2004.01 (2003).

\bibitem{Conn2009Introduction}
A. Conn, K. Scheinberg, L. Vicente: Introduction to Derivative-Free Optimization, Society for Industrial and Applied Mathematics, 2009. \url{https://doi.org/10.1137/1.9780898718768}.

\bibitem{Guzman2015lower}
G. Cristóbal, A. Nemirovski: On lower complexity bounds for large-scale smooth convex optimization. Journal of Complexity 31.1 (2015): 1--14.

\bibitem{Cohen2018Acceleration}
M. B. Cohen, J. Diakonikolas, L. Orecchia: On Acceleration with Noise-Corrupted Gradients, International Conference on Machine Learning, 2018.

\bibitem{Devolder2013smooth}
O. Devolder, F. Glineur, Y. Nesterov: First-order methods of smooth convex optimization with inexact oracle. Math. Prog., 146: 37–75, 2013.

\bibitem{Devolder2013strongly}
O. Devolder, F. Glineur, Y. Nesterov: First-order methods with inexact oracle: the strongly convex case. CORE Discussion Papers 2013016 (2013): 47. \url{https://optimization-online.org/2014/03/4280/}

\bibitem{Devolder2013disser}
O. Devolder: Exactness, inexactness and stochasticity in first-order methods for large-scale convex optimization. Diss. CORE UCLouvain, 2013.

\bibitem{Dvinskikh2020Accelerated}
D. M. Dvinskikh, A. I. Tyurin, A. V. Gasnikov, and C. C. Omel’chenko: Accelerated and Unaccelerated Stochastic Gradient Descent in Model Generality. Math Notes 108, 511--522, 2020.

\bibitem{Dvurechensky2022gradient}
P. Dvurechensky: A gradient method with inexact oracle for composite nonconvex optimization, Computer Research and Modeling, 14(2), 2022.

\bibitem{Dvurechensky2016Stochastic}
P. Dvurechensky, A. Gasnikov: Stochastic intermediate gradient method for convex problems with stochastic inexact oracle, Journal of Optimization Theory and Applications, 171(1): 121–145, 2016.

\bibitem{Duchi2011Adaptive}
J. Duchi, E. Hazan, Y. Singer: Adaptive subgradient methods for online learning and stochastic optimization. Journal of Machine Learning Research, 12(Jul.): 2121--2159, 2011.

\bibitem{drusvyatskiy2019nonsmooth}
D. Drusvyatskiy, A. D. Ioffe, A. S. Lewis: Nonsmooth optimization using taylor-like models: error bounds, convergence, and termination criteria. Math. Program. 185, 357–383 (2021). \url{https://doi.org/10.1007/s10107-019-01432-w} 

\bibitem{Aspremont2008Smooth}
A. d’Aspremont: Smooth optimization with approximate gradient, SIAM Journal on Optimization, 19(3): 1171–1183, 2008.

\bibitem{frank1956algorithm}
M. Frank, P. Wolfe: An algorithm for quadratic programming, NavalRes. Logist. Q. 3 (1956), pp. 95–110.

\bibitem{Gasnikov2019fast}
A. V. Gasnikov, A. I. Tyurin: Fast Gradient Descent for Convex Minimization Problems with an Oracle Producing a $(\delta, L)$-Model of Function at the Requested Point. Computational Mathematics and Mathematical Physics, 2019, Vol. 59, No. 7, pp. 1085--1097.

\bibitem{Kabanikhin2011}
S. I. Kabanikhin: Inverse and ill-posed problems: theory and applications, Vol. 55, Walter De Gruyter, 2011.

\bibitem{lu2018relatively}
H. Lu, R. M. Freund, Y. Nesterov: Relatively smooth convex optimization by first-order methods, and applications, SIAM. J. Optim. 28 (2018), pp. 333–354.

\bibitem{mairal2013optimization}
J. Mairal: Optimization with first-order surrogate functions, International Conference on Machine Learning, Proceedings of Machine Learning Research, PMRL, Atlanta, 2013, pp. 783–791.

\bibitem{Gasnikov2017Hilbert}
V. Matyukhin, S. Kabanikhin, M. Shishlenin, N. Novikov, A. Vasin,  A. Gasnikov: Convex optimization with inexact gradients in Hilbert space and applications to elliptic inverse problems. In International Conference on Mathematical Optimization Theory and Operations Research. Cham: Springer International Publishing. pp. 159--175, (2021).

\bibitem{Nabou2024qoracle}
Y. Nabou, F. Glineur, I. Necoara: Proximal gradient methods with inexact oracle of degree $q$ for composite optimization. arXiv:2401.10624v1 \url{https://arxiv.org/pdf/2401.10624v1.pdf}

\bibitem{Nemirovskii1983Complexity}
A. Nemirovskii, D. Yudin: Problem Complexity and Method Efficiency in Optimization. Wiley, New York (1983)

\bibitem{Nesterov2015Universal}
Y. Nesterov: Universal gradient methods for convex optimization problems, Math. Program. 152 (2015), pp. 381–404.

\bibitem{nesterov2006cubic}
Y. Nesterov and B. Polyak: Cubic regularization of Newton method and its global performance, Math. Program. 108 (2006), pp. 177–205.

\bibitem{Nesterov1983}
Y. Nesterov: A method for unconstrained convex minimization with the rate of convergence of $O(1/k^2)$. Doklady AN SSSR 269, 543–547 (1983)

\bibitem{nesterov2013gradient}
Y. Nesterov: Gradient methods for minimizing composite functions. Math. Program. 140, 125–161 (2013). \url{https://doi.org/10.1007/s10107-012-0629-5}

\bibitem{Nesterov_book}
Y. Nesterov: Lectures on convex optimization. Switzerland: Springer Optimization and Its Applications, 2018.

\bibitem{Nesterov1988approach}
Y. Nesterov: On an approach to the construction of optimal methods of minimization of smooth convex function. Ekonom. i. Mat. Metody (In Russian) 24, 509–517 (1988).

\bibitem{Nesterov2005Smooth}
Y. Nesterov: Smooth minimization of non-smooth functions. Math. Program. 103, 127–152 (2005). \url{https://doi.org/10.1007/s10107-004-0552-5}

\bibitem{ochs2017non}
P. Ochs, J. Fadili, T. Brox: Non-smooth non-convex Bregman minimization: unification and new algorithms. J Optim Theory Appl 181, 244–-278 (2019). \url{https://doi.org/10.1007/s10957-018-01452-0}

\bibitem{Polyak1987book}
B. T. Polyak: Introduction to Optimization. Optimization Software, New York (1987)

\bibitem{Risteski2016Algorithms}
A. Risteski, Y. Li: Algorithms and matching lower bounds for approximately-convex optimization, Advances in Neural Information Processing Systems 29 (2016), pp. 4745-4753.

\bibitem{Shwartz2014Understanding}
Sh. Shalev-Shwartz, Sh. Ben-David: Understanding machine learning: from theory to algorithms. Cambridge University Press, 2014.

\bibitem{Stonyakin2021OMS}
F. Stonyakin, A. Tyurin, A. Gasnikov, P. Dvurechensky, A. Agafonov, D. Dvinskikh, M. Alkousa, D. Pasechnyuk, S. Artamonov, V. Piskunova: Inexact model: a framework for optimization and variational inequalities. Optimization Methods and Software Volume 36, 2021 - Issue 6, Pages 1155--1201. \url{https://doi.org/10.1080/10556788.2021.1924714}

\bibitem{FedorRolandbook}
E. Vorontsova, R. Hildebrand, A. Gasnikov, F. Stonyakin: Convex optimization. arXiv:2106.01946 \url{https://arxiv.org/pdf/2106.01946.pdf} (in Russian)

\bibitem{Zou2018Deep}
D. Zou, Y. Cao, D. Zhou, Q. Gu: Stochastic Gradient Descent Optimizes Over-parameterized Deep ReLU Networks. Mach Learn 109, 467–492 (2020). \url{https://doi.org/10.1007/s10994-019-05839-6}

\end{thebibliography}
\end{document}